\documentclass[11pt]{amsart}

\usepackage{amsfonts, amssymb}
\usepackage{amscd, amsxtra, amsmath}
\usepackage{amsbsy}
\usepackage{epsfig}
\usepackage{graphicx}
\usepackage{psfrag}
\usepackage{amsmath}
\usepackage{amsthm}
\usepackage{setspace}
\usepackage{url}
\usepackage{color}
\usepackage{hyperref}
\usepackage{enumerate}
\usepackage{enumitem}

\theoremstyle{plain}
\newtheorem{theorem}{Theorem}[section]
\newtheorem{lemma}[theorem]{Lemma}
\newtheorem{proposition}[theorem]{Proposition}
\newtheorem{corollary}[theorem]{Corollary}

\theoremstyle{definition}
\newtheorem{remark}[theorem]{Remark}

\newcommand{\MM}{\mathcal M}

\newcommand{\calA}{\mathcal A}
\newcommand{\calB}{\mathcal B}
\newcommand{\calC}{\mathcal C}
\newcommand{\calD}{\mathcal D}
\newcommand{\BM}{\overline{\mathcal M}}

\newcommand{\PP}{\mathbb P}

\newcommand{\OO}{\mathcal O}

\newcommand{\RR}{\mathbb R}

\newcommand{\HD}{\widehat{\Delta}}
\newcommand{\TD}{\widetilde{\Delta}}

\newcommand{\TMM}{\widetilde{\mathcal M}}

\newcommand{\Eff}{\operatorname{Eff}}

\newcommand{\BEff}{\overline{\operatorname{Eff}}}

\newcommand{\ps}{\operatorname{ps}}

\newcommand{\bbC}{\mathbb C}

\newcommand{\bbQ}{\mathbb Q}
\newcommand{\bbR}{\mathbb R}

\newcommand{\bbP}{\mathbb P}

\newcommand{\Int}{\operatorname{Int}}
\newcommand{\sat}{\operatorname{sat}}

\title{Extremal higher codimension cycles on moduli spaces of curves}

\date{\today}

\author{Dawei Chen}
\address{Department of Mathematics, Boston College, Chestnut Hill, MA 02467}
\email{dawei.chen@bc.edu}

\author{Izzet Coskun}
\address{Department of Mathematics, Statistics, and Computer Science, University of Illinois at Chicago, Chicago, IL 60607}
\email{coskun@math.uic.edu}

\thanks{During the preparation of this article the first author was partially supported by the NSF grant DMS-1200329, the NSF CAREER grant
DMS-1350396, and the second author was partially supported by the NSF CAREER grant DMS-0950951535.}

\begin{document}

\begin{abstract}
We show that certain geometrically defined higher codimension cycles are extremal in the effective cone of the moduli space $\BM_{g,n}$ of stable genus $g$ curves with $n$ ordered marked points. In particular, we prove that codimension two boundary strata are extremal and exhibit extremal boundary strata of  higher codimension. We also show that the locus of hyperelliptic curves with a marked Weierstrass point in $\BM_{3,1}$ and the locus of hyperelliptic curves in $\BM_4$ are extremal cycles. In addition, we exhibit infinitely many extremal codimension two cycles in $\BM_{1,n}$ for $n\geq 5$ and in $\BM_{2,n}$ for $n\geq 2$. 
\end{abstract}

\maketitle
\setcounter{tocdepth}{1}
\tableofcontents

\section{Introduction}
\label{sec:intro}

Let $\BM_{g,n}$ denote the Deligne-Mumford-Knudsen moduli space of stable genus $g$ curves with $n$ ordered marked points. In this paper, we study the effective cones of higher codimension cycles on $\BM_{g,n}$. 

Motivated by the problem of determining the Kodaira dimension of $\BM_{g,n}$, the cone of effective divisors of $\BM_{g,n}$ has been studied extensively, see e.g.~\cite{HarrisMumfordKodaira, HarrisKodaira, EisenbudHarrisKodaira, FarkasKoszul, LoganKodaira, Vermeire, CastravetTevelev, ChenCoskun}. In contrast, little is known about higher codimension cycles on $\BM_{g,n}$, in part because their positivity properties 
are not as well-behaved. For instance, higher codimension nef cycles may fail to be pseudoeffective \cite{DELV}. Furthermore, unlike the case of divisors, we lack simple numerical, cohomological and  geometric conditions for determining whether higher codimension cycles are nef or pseudoeffective. 

In this paper, we show that certain geometrically defined higher codimension cycles span extremal rays of the effective cone of $\BM_{g,n}$. 

\subsection*{Main Results} \begin{enumerate}[label={\upshape(\roman*)}]

\item Every codimension two boundary stratum of $\BM_g$ and of $\BM_{0,n}$ is extremal (Theorems~\ref{thm:eff3}, \ref{thm:eff2mg} and \ref{thm:eff0n}). 

\item The higher codimension  boundary strata associated to certain dual graphs described in \S \ref{sec:genus-g} are extremal in $\BM_g$  (Theorem \ref{thm:arbitrarily-high}). 

\item Every codimension $k$ boundary stratum of $\BM_{0,n}$ parameterizing curves with $k$ marked tails attached to an unmarked $\PP^1$ is extremal (Theorem~\ref{thm:unmarked}). 

\item There exist infinitely many extremal effective codimension two cycles in $\BM_{1,n}$ for every $n\geq 5$ (Theorem~\ref{thm:eff1n}) and in $\BM_{2,n}$ for every $n\geq 2$ (Theorem~\ref{thm:eff2n}).

\item The locus of hyperelliptic curves with a marked Weierstrass point is a non-boundary extremal codimension two cycle in $\BM_{3,1}$ (Theorem~\ref{thm:hw}).  

\item The locus of hyperelliptic curves is a non-boundary extremal codimension two cycle in $\BM_4$ (Theorem~\ref{thm:hyp-4}). 
\end{enumerate}

These results illustrate that the effective cone of higher codimension cycles on $\BM_{g,n}$ can be very complicated even for small values of $g$ and $n$. 

In order to verify the extremality of a codimension $k$ cycle, we use two criteria. First,  we find a criterion that shows the extremality of loci that drop the largest possible dimension under a morphism (Proposition~\ref{prop:extremal}). We then apply the criterion to morphisms from $\BM_{g,n}$ to different modular compactifications of $\MM_{g,n}$. Second, we use induction on dimension. To prove that a cycle $Z$ is extremal, we first show that  $Z$ is extremal in a divisor $D\subset \BM_{g,n}$ containing $Z$, and then show that effective cycles representing $Z$ must be contained in $D$ (Proposition~\ref{prop:subvariety}). 

The paper is organized as follows. In Sections~\ref{sec:prelim} and~\ref{sec:prelim-moduli} we review the basic properties of effective cycles and moduli spaces of curves. Then we carry out the study of the effective cones of $\BM_{g,n}$ according to the values of $g$ and $n$: $g=3$ and $n\leq 1$ (Section~\ref{sec:genus-3});  
$g\geq 4$ and $n=0$ (Section~\ref{sec:genus-g}); $g=0$ and arbitrary $n$ (Section~\ref{sec:eff0n}); $g=1$ and $n\geq 5$ (Section~\ref{sec:eff1n}); 
$g=2$ and $n\geq 2$ (Section~\ref{sec:eff2n}). 

\subsection*{Acknowledgments} We would like to thank Maksym Fedorchuk, Joe Harris, Brian Lehmann, John Lesieutre, Anand Patel, Luca Schaffler  and Nicola Tarasca for helpful discussions related to this paper.  

\section{Preliminaries on effective cycles}
\label{sec:prelim}

In this section, we review basic properties of the effective cone of cycles on an algebraic variety and develop some criteria for proving the extremality of an effective cycle. Throughout the paper, all varieties are defined over $\bbC$ and all linear combinations of cycles are with $\bbR$-coefficients. 

Let $X$ be a complete variety. A \emph{cycle} on $X$ is a formal sum of subvarieties of $X$.  A cycle is \emph{$k$-dimensional} if all subvarieties in the sum are $k$-dimensional. A cycle is \emph{effective} if all coefficients in the sum 
are nonnegative. Two $k$-dimensional cycles $A$ and $B$ on $X$ are  {\em numerically equivalent}, if $A \cap P = B \cap P$ under the degree map for 
all weight $k$ polynomials $P$ in Chern classes of vector bundles on $X$, where $\cap$ is the cap product, see \cite[Chapter 19]{FultonIntersection}. When $X$ is nonsingular, this is equivalent to requiring $A\cdot C = B\cdot C$ for all subvarieties $C$ of codimension $k$, where $\cdot$ is the intersection product. The focus of the paper is the moduli space of curves, which is $\bbQ$-factorial but may have finite quotient singularities. Nevertheless, the intersection product is still compatible with the cup product, see \cite[Section 1]{EdidinChow}. 
  
Denote by $[Z]$ the \emph{numerical class} of a cycle $Z$. Let $N_k(X)$ (resp. $N^k(X)$) denote the $\RR$-vector space of cycles of 
dimension $k$ (resp. codimension $k$) modulo numerical equivalence. It is a finite dimensional vector space. Let $\Eff_k(X) \subset N_k(X)$ (resp. $\Eff^k(X) \subset N^k(X)$) denote the {\em effective cone} of dimension $k$ (resp. codimension $k$) cycles generated by all effective cycle classes. Their closures $\BEff_k(X)$ and $\BEff^k(X)$ are called the {\em pseudoeffective cones}.

A convex cone is most conveniently described by specifying its extremal rays. Recall that a ray $R$ is called {\em extremal}, if for every $D \in R$ and $D= D_1 + D_2$ with $D_1, D_2$ in the cone, we have $D_1, D_2 \in R$. 
If an extremal ray is spanned by the class of an effective cycle $D$, we say that $D$ is an \emph{extremal effective cycle}. 

There is a well-developed theory to study the cone $\Eff^1(X)$, including numerical, cohomological, analytic and geometric conditions for checking whether a divisor is in $\BEff^1(X)$, see \cite{LazarsfeldPositivity}. In contrast, $\Eff^k(X)$ for $k\geq 2$ is not well-understood. We first describe two simple criteria for checking extremality of cycles in $\Eff^k(X)$. 

Let $f: X\to Y$ be a morphism between two complete varieties. To a subvariety $Z\subset X$ of dimension $k$ we associate an index 
$$e_f(Z) = \dim Z - \dim f(Z). $$
Note that $e_f(Z) > 0$ if and only if $Z$ drops dimension under $f$.   

\begin{proposition}
\label{prop:main-tool}
Let $f: X\to Y$ be a morphism between two projective varieties and let $k > m \geq 0$ be two integers. Let $Z$ be a $k$-dimensional subvariety of $X$ such that $e_f(Z) \geq k-m$. If $[Z] = a_1[Z_1] + \cdots + a_r[Z_r] \in N_k(X)$, where $Z_i$ is a $k$-dimensional subvariety of $X$ and $a_i > 0$ for all $i$, then $e_f(Z_i) \geq k-m$ for every $1 \leq i \leq r$. 
\end{proposition}

\begin{proof}
Let $A$ and $B$ be two very ample divisor classes on $X$ and on $Y$, respectively. Then $N = f^{*}B$ is base-point-free on $X$. In particular, if $U\subset X$ is an effective cycle of dimension $k$, then the intersection  $N^{j} \cdot [U]$ is either zero or can be represented by an effective cycle of dimension $k - j$ for $1\leq j \leq k$. In the latter case, by the projection formula and the very ampleness of $A$, we have $0< A^{k-j} \cdot N^j \cdot [U] = B^j \cdot [f_*(W)]$, where $W$ is an effective cycle representing $A^{k-j} \cdot [U]$. Since $e_f(Z) \geq k-m$, we conclude that $N^{m+1} \cdot [Z] =0$. Therefore, $N^{m+1} \cdot [Z_i] = 0$ for $1 \leq i \leq r$. Otherwise, $N^{m+1} \cdot [Z_i]$ can be represented by an effective cycle. Intersecting both sides of the equality $N^{m+1} \cdot [Z] = N^{m+1} \cdot (\sum_{i=1}^r a_i[Z_i])$ with $A^{k-m-1}$ we would obtain a contradiction. By the projection formula, we conclude that $B^{m+1} \cdot (f_*[Z_i]) = 0$, hence  $\dim f(Z_i) \leq m$ as desired. 
\end{proof}

As a corollary of Proposition \ref{prop:main-tool}, we obtain the following extremality criterion.

\begin{proposition}
\label{prop:extremal}
Let $f: X\to Y$ be a morphism between two projective varieties. Fix two integers $k > m \geq 0$.  
Among all $k$-dimensional subvarieties $Z$ of $X$, assume that only finitely many of them, denoted by $Z_1, \ldots, Z_n$, 
satisfy  $e_f(Z) \geq k-m$. If the classes of $Z_1,\ldots, Z_n$ 
are linearly independent, then each $Z_i$ is an extremal effective cycle in $\Eff_k(X)$. 
\end{proposition}

\begin{proof}
Suppose that $[Z_i] = a_1[D_1] + \cdots + a_r[D_r]$ with $D_j$ a $k$-dimensional subvariety of $X$ and $a_j >0$ for all $j$. Then, by Proposition \ref{prop:main-tool}, $e_f(D_j)\geq k-m$. Since $Z_1, \dots, Z_n$ are the only $k$-dimensional 
 subvarieties of $X$ with index $e_f \geq k-m$, we conclude that $D_j$ has to be one of $Z_1, \dots, Z_n$.  Since their classes are independent, we conclude that $[D_j]$ is proportional to $[Z_i]$ for all $j$. Therefore, $Z_i$ is extremal in $\Eff_k(X)$. 
\end{proof}

We will apply Propositions \ref{prop:main-tool} and \ref{prop:extremal} to morphisms from $\BM_{g,n}$ to alternate modular compactifications of $\MM_{g,n}$. 
We also remark that the classes of exceptional loci in higher codimension may fail to be linearly independent. 

\begin{remark}
The following example, which is related to Hironaka's example of a complete but non-projective variety  (see \cite[Appendix B, Example 3.4.1]{Hartshorne}),
was pointed out to the authors by John Lesieutre. Let $X$ be a smooth threefold, and let $C, D\subset X$ be two smooth curves meeting transversally at two points $p$ and $q$. Let $Y$ be the blowup of $X$ along $C$. The fibers $F_r$ of the exceptional locus over $C$
have the same numerical class for all $r\in C$. Next, let $Z$ be the blowup of $Y$ along the proper transform of $D$.  The proper transforms $E_p$ of $F_p$ and $E_q$ of $F_q$ have classes different from (the transform of) $F_r$ for $r\neq p, q$. However, $E_p$ and $E_q$ have the same class, and their normal bundles are isomorphic to $\OO(-1)\oplus \OO(-1)$. Hence, they are flopping curves if we blow up $D$ first and then blow up the proper transform of $C$. In particular, there exists a small contraction $Z\to W$ such that the exceptional locus consists of $E_p$ and $E_q$ only. 
\end{remark}

The next corollary will be useful when studying the effective cones of the boundary strata.

\begin{corollary}\label{cor:prod}
Let $X$ and $Y$ be projective varieties and let $Z$ be an extremal effective cycle of codimension $k$ in $X$. Then $Z \times Y$ is an extremal effective cycle of codimension $k$ in $X \times Y$.
\end{corollary}

\begin{proof}
Suppose $[Z \times Y] = \sum a_i [U_i]$, where $a_i >0$ and $U_i$ are irreducible subvarieties of $X\times Y$. Let $\pi$ denote the projection of $X\times Y$ onto $X$. Then, by Proposition \ref{prop:main-tool}, $e_{\pi}(U_i) = \dim (Y)$. Hence, $U_i = V_i \times Y$ for a subvariety $V_i \subset X$ for each $i$. By the projection formula, $[Z] = \sum a_i [V_i]$. Since $Z$ is extremal in $X$, the classes of $Z$ and $V_i$ are proportional. Hence, their pullbacks to $X \times Y$ are proportional. We conclude that $[U_i]$ is proportional to $[Z \times Y]$ for every $i$ and $Z \times Y$ is extremal in $X \times Y$. 
\end{proof}

Another useful criterion is the following. Let $A_k(X)$ (resp. $A^k(X)$) denote the \emph{Chow group} of rationally equivalent cycle classes of dimension $k$ (resp. codimension $k$) in $X$ with $\bbR$-coefficients. 
If two cycles are rationally equivalent, they are also numerically equivalent, hence we have a map $A_k(X)\to N_k(X)$. 
For a morphism $Y\to X$, there is a push-forward map $A_k(Y)\to A_k(X)$. 

\begin{proposition}
\label{prop:subvariety}
Let $\gamma: Y\to X$ be a morphism between two projective varieties. Assume that $A_k(Y)\to N_k(Y)$ is an isomorphism and that the composite 
$\gamma_{*}: A_k(Y) \to A_k(X)\to N_k(X)$ is injective. Moreover, assume that $f: X\to W$ is a morphism to a projective variety $W$ whose exceptional locus is contained in $\gamma(Y)$. If a $k$-dimensional subvariety $Z\subset Y$ is an extremal cycle in $\Eff_k(Y)$ and if $e_f(\gamma(Z)) > 0$, then $\gamma(Z)$ is also extremal in $\Eff_k(X)$.  
\end{proposition}

\begin{proof}
Suppose that $[\gamma(Z)] = a_1[Z_1] + \cdots + a_r [Z_r]\in N_k(X)$ for subvarieties $Z_i\subset X$ with $a_i > 0$ for all $i$. Since 
$e_f(\gamma(Z)) > 0$, by Proposition~\ref{prop:main-tool}, $Z_i$ drops dimension under $f$.  Hence, $Z_i$ is contained in $\gamma(Y)$ for all $i$. 
By assumption, $[Z] = a_1[Z'_1] + \cdots + a_r [Z'_r]$  in $N_k(Y)$, where $Z'_i\subset Y$ such that $\gamma_{*}[Z'_i] = [Z_i]$. 
Since $Z$ is extremal in $\Eff_k(Y)$, $[Z'_i]$ has to be proportional to $[Z]$ as cycle classes in $Y$. Since $\gamma_{*}$ is injective, $[Z_i] = \gamma_{*}[Z'_i]$ is proportional to $\gamma_{*}[Z] $ as cycle classes in $X$ as well.   
\end{proof}

We will apply Proposition~\ref{prop:subvariety} to the case when $X$ is the moduli space of curves and $Y$ is an irreducible boundary divisor. Occasionally, we will be able to prove that effective cycles expressing a cycle class $[Z]$ are contained in a union of divisors.  We will then use the following technical result to deduce the extremality of $[Z]$.  

\begin{proposition}
\label{prop:union}
Let $Y = \bigcup_{i=0}^n D_i$ be a union of irreducible divisors $D_i\subset X$ such that $D_i \cap D_j$ consists of $m_{i,j}$ irreducible codimension two  subvarieties 
$D_{i,j,k}$ for $0\leq i < j \leq n$ and $k=1, \ldots, m_{i,j}$. Assume for all $i, j, k$ that $D_{i,j,k}$ is extremal in $D_i$ and that 
a linear combination $$\sum_{j=0\atop j\neq i}^{n}\sum_{k=1}^{m_{i,j}} a_{i,j,k} [D_{i,j,k}]$$ is effective in $D_i$ if and only if $a_{i,j,k}\geq 0$. Further assume that $A^1(D_i) \to N^1(D_i)$ is an isomorphism and $A^1(Y) \to A^2(X) \to N^2(X)$ are injective. Let $Z\subset D_0$ be an effective divisor. Finally, assume that for every effective expression 
$$[Z]= \sum_{i} a_i [Z_i]  \in \Eff^2(X),$$
where $a_i > 0$ and $Z_i\subset X$ is a codimension two subvariety,  $Z_i$ is contained in $Y$ for all $i$. If $Z$ is extremal in $\Eff^1(D_0)$, then $Z$ is also extremal in $\Eff^2(X)$. 
\end{proposition} 

\begin{proof}
Suppose that 
\begin{eqnarray}
\label{eq:union-relation}
[Z]= \sum_i a_i [Z_i]  \in \Eff^2(X)
\end{eqnarray}
for $a_i > 0$ and $Z_i\subset X$ irreducible codimension two subvariety. We want to show that $[Z_i]$ is proportional to $[Z]$ in $N^2(X)$. 
By assumption, $Z_i$ is contained in $Y$ for all $i$. Reexpress the summation 
in \eqref{eq:union-relation} as  
$$S_0 + \cdots + S_n, \ \ \mbox{where} \ \ S_j = \sum_{\substack{Z_i \subset D_j, \\ Z_i \not\subset D_k  \  \mbox{for} \ k <j}} a_i [Z_i].$$  

By assumption, the second map in 
$$ \bigoplus_{i=1}^n N^1(D_i) \to N^1(Y) \to N^2(X) $$  
is injective. The kernel of the first map is generated by elements of type 
$$ (0, \cdots, [D_{i,j,k}], \cdots, - [D_{i,j,k}], \cdots, 0 ) $$
for $0\leq i < j \leq n$ and $1\leq k \leq m_{i,j}$, where the nonzero entries occur in the $i$th and $j$th places. Lifting \eqref{eq:union-relation} to the direct sum, there exist $a_{i,j,k}\in \bbR$ 
such that 
\begin{eqnarray}
\label{eq:D0}
[Z] = S_0 + \sum_{j=1}^n\sum_{k=1}^{m_{0,j}} a_{0,j,k} [D_{0,j,k}] \in N^1(D_0)
\end{eqnarray}
and 
\begin{eqnarray}
\label{eq:Di}
0 = S_i + \sum_{j=0}^{i-1}\sum_{k=1}^{m_{j,i}} (-a_{j,i,k}) [D_{j,i,k}] + \sum_{j=i+1}^n\sum_{k=1}^{m_{i,j}} a_{i,j,k} [D_{i,j,k}] \in N^1(D_i) 
\end{eqnarray}
for $1\leq i \leq n$.  

In the expressions, each $a_{i,j,k}$ appears twice with opposite signs. Since $S_i$ is effective, our assumption that the coefficients of $D_{i,j,k}$ in an effective sum have to be positive along with \eqref{eq:Di} implies that $a_{j,i,k} \geq 0$ for $0\leq j < i\leq n $ and $a_{i,j,k} \leq 0$ for 
$1\leq i < j \leq n$. Therefore, we conclude that $a_{i,j,k} = 0$ for $1\leq i < j \leq n$ and all $k$. Now \eqref{eq:Di} reduces to 
\begin{eqnarray}
\label{eq:Di0}
 S_i = \sum_{k=1}^{m_{0,j}} a_{0,j,k} [D_{0,i,k}] \in N^1(D_i) 
\end{eqnarray}
for $1\leq i \leq n$.  
First, assume that $[Z]$ is not proportional to any $[D_{0,j,k}]$ in $N^1(D_0)$. By assumption that $Z$ is extremal in $D_0$ and $a_{0,i,k} \geq 0$ for $i > 0$,  \eqref{eq:D0} and \eqref{eq:Di0} imply that $a_{0,j,k} = 0$ for all $j, k$, $S_i=0$ and  the classes $[Z_i]$ in $S_0$ are proportional to $[Z]$. The classes 
$[D_{0,j,k}]$ are independent in $D_0$, otherwise it would contradict the assumption that $\sum_{j=1}^n\sum_{k=1}^{m_{0,j}} a_{0,j,k} [D_{0,j,k}]$ is effective in $D_0$ if and only if $a_{0,j,k}\geq 0$. Therefore, we conclude that the composite 
$N^1(D_0) \to N^1(Y) \to N^2(X)$ is injective, and hence $Z$ is also extremal in $\Eff^2(X)$. If $[Z]$ is proportional to some $[D_{0,j,k}]$ in $N^1(D_0)$, then by \eqref{eq:D0}  and \eqref{eq:Di0}, $a_{0,j',k'} = 0$ for all $(j', k')\neq (j,k)$ and classes in the summation $S_i$ are proportional to $[Z]$ in $N^1(D_i)$ for all $i \geq 0$. We still conclude that  $Z$ is extremal in $\Eff^2(X)$. 
\end{proof}

\section{Preliminaries on moduli spaces of curves}
\label{sec:prelim-moduli}

Let $\BM_{g,n}$ be the moduli space of stable genus $g$ curves with $n$ \emph{ordered} marked points.  
The boundary $\Delta$ of $\BM_{g,n}$ consists of irreducible boundary divisors $\Delta_0$ and $\Delta_{i;S}$ for $1\leq i \leq [g/2]$, 
$S\subset \{1,\ldots,n \}$ such that $|S|\geq 2$ if $i =0$. A general point of $\Delta_0$ parameterizes an irreducible nodal curve 
of geometric genus $g-1$. A general point of $\Delta_{i;S}$ parameterizes a genus $i$ curve containing the marked points labeled by $S$, 
attached at one point to a genus $g-i$ curve containing the marked points labeled by the complement $S^c$. We use $\lambda$ to denote the first Chern class of the Hodge bundle. Let $\psi_i$ be the first Chern class of the cotangent line bundle associated to the $i$th marked point and  let $\psi= \sum_{i=1}^n \psi_i$.  

We may also consider curves with unordered marked points. Let $\TMM_{g,n} = \BM_{g,n}/\mathfrak{S}_n$ be the moduli space of stable genus $g$ curves with $n$ \emph{unordered} marked points. The boundary of $\TMM_{g,n}$ consists of irreducible boundary divisors $\TD_0$ and $\TD_{i;k}$ for $1\leq i \leq [g/2]$ and $0\leq k\leq n$ such that $k\geq 2$ if $i = 0$. 

Moduli spaces of curves of lower genera can be glued together to form the boundary of $\BM_{g,n}$. 
Set $ \HD_0 = \BM_{g-1, n+2}/\mathfrak{S}_2$, where $\mathfrak{S}_2$ interchanges the last two marked points. Identifying the last two marked points to form a node 
induces a gluing morphism 
$$\alpha_0: \HD_0\to \Delta_0. $$   
For $0< i < g/2$ or $i =g/2$ if $g$ is even and $n>0$, denote by $\HD_{i;S} = \BM_{i, |S| + 1} \times \BM_{g-i, n- |S|+1}$. 
Identifying the last marked points in the two factors induces   
$$ \alpha_{i;S}: \HD_{i;S} \to \Delta_{i; S}. $$
For $g$ even, $i = g/2$ and $n=0$, set $\HD_{g/2} = (\BM_{g/2, 1}\times \BM_{g/2,1})/\mathfrak{S}_2$ and 
$$ \alpha_{g/2}: \HD_{g/2}\to \Delta_{g/2}$$ is induced by identifying the two marked points to form a node. Restricted to the complement of the locus of curves with more than one node, the gluing morphisms are isomorphisms. Later on when studying codimension two 
boundary strata of $\BM_{g,n}$, this will help us identify them with boundary divisors on moduli spaces of curves of lower genera. 

In order to study extremal higher codimension cycles on $\BM_{g,n}$ and $\TMM_{g,n}$, we need to use the fact that the codimension one boundary strata are extremal. This is well-known (see e.g. \cite[1.4]{RullaThesis}), and we explain it briefly as follows. 

\begin{proposition}
\label{prop:boundary-extremal}
Every irreducible boundary divisor is extremal on $\BM_{g,n}$ and $\TMM_{g,n}$. 
\end{proposition}

\begin{proof}
In general, one can exhibit a moving curve in a boundary divisor that has negative intersection with the divisor, which implies that it is extremal by \cite[Lemma 4.1]{ChenCoskun}.  For $\Delta_0$, fix a curve $C$ of genus $g-1$ with $n+1$ distinct fixed points $p_1, \dots, p_{n+1}$. One obtains the desired moving curve by gluing  a varying point on $C$ to $p_{n+1}$. For $\Delta_{i;S}$, fix a curve $C$ of genus $i$ with distinct marked points labeled by $S$ and a curve $C'$ of genus $g-i$ with distinct marked points labeled by $S^c$. One obtains the desired moving curve by gluing a fixed point on $C'$ distinct from the previously chosen points to a varying point on $C$. The only exceptional case is $\Delta_0$ in $g=1$. Its class equals $12\lambda$, which is semi-ample. However, $\lambda$ induces a fibration over $\BM_{1,1}$, hence it spans an extremal ray in both the nef cone and the effective cone.   
\end{proof}

We also need to use several other compactifications of $\MM_{g,n}$. Let $\tau: \BM_g\to \calA^{\sat}_g$ denote the \emph{extended Torelli map} from $\BM_g$ to the Satake compactification of the moduli space $\calA_g$ of $g$-dimensional principally polarized abelian varieties (see e.g. \cite[III. 16]{BHPV}). It maps a stable curve to the product of the Jacobians of the irreducible components of its normalization. The exceptional locus of $\tau$ is the total boundary $\Delta$. 

Let $\ps: \BM_{g,n} \to \BM_{g,n}^{\ps}$ be the \emph{first divisorial contraction} of the log minimal model program for $\BM_{g,n}$ (\cite{HassettHyeonContraction, AFSV}), where $\BM_{g,n}^{\ps}$ is the moduli space of genus $g$ \emph{pseudostable} curves with $n$ ordered marked points. It contracts $\Delta_{1; \emptyset}$ only, replacing an unmarked elliptic tail by a cusp. 

Let $\calA = \{ a_1, \ldots, a_n   \}$ be a collection of $n$ rational numbers such that $0< a_i \leq 1$ for all $i$ and $2g-2 +\sum_{i=1}^n a_i > 0$. The moduli space of {\em weighted} stable curves
$\BM_{g,\calA}$ parameterizes the data $(C, p_1, \dots, p_n, \calA)$ such that 
\begin{itemize}
\item $C$ is a connected, reduced, at-worst-nodal arithmetic genus $g$ curve. 
\item  $p_1, \ldots, p_n$ are smooth points of $C$ assigned the weights $a_1, \ldots, a_n$, respectively, where the total weight of any points that coincide is at most one.
\item   for every irreducible component $X$ of $C$, the divisor class $K_C + \sum_{i=1}^n a_i p_i$ is ample restricted to $X$, i.e. numerically 
$$2g_X - 2 +\#(X\cap \overline{C\backslash X}) + \sum_{p_i\in X} a_i > 0,$$ where $g_X$ is the arithmetic genus of $X$. 
\end{itemize}
If $\calA = \{ 1, \ldots, 1 \}$, then $\BM_{g,\calA} =\BM_{g,n}$. Hassett constructs a morphism $f_{\calA}: \BM_{g,n} \to \BM_{g, \calA}$ that modifies the locus of curves violating the stability inequality in the above \cite{HassettWeight}. Note that if $g_X \geq 1$ or $\#(X\cap \overline{C\backslash X})\geq 2$ for $C$ parameterized by $\BM_{g,n}$, 
the above inequality always holds for any permissible $\calA$. Hence, the exceptional locus of $f_{\calA}$ consists of curves that have a rational tail with at least three marked points of total weight at most one.

\section{The effective cones of $\BM_{3}$ and $\BM_{3,1}$}
\label{sec:genus-3}

We want to study the effective cone of higher codimension cycles on $\BM_{g,n}$. 
First, consider the case of codimension two and $n=0$. Since $\dim \BM_2 = 3$, a cycle of codimension two on $\BM_2$ is a curve, and the cone of curves of $\BM_2$ is known to be spanned by the one-dimensional topological strata ($F$-curves). In this section, we introduce our methods by studying the first interesting case $g=3$ in detail. In later sections, we will generalize some of the results to $g\geq 4$ or $n > 0$.  

The Chow ring of $\BM_3$ was computed in \cite{FaberChow}.  For ease of reference, we preserve Faber's notation introduced in \cite[p. 340--343]{FaberChow}: 

\begin{itemize}
\item Let $\Delta_{00}\subset \BM_3$ be the closure of the locus parameterizing irreducible curves with two nodes. 

\item Let $\Delta_{01a}\subset \BM_3$ be the closure of the locus parameterizing a rational nodal curve attached to a genus two curve at one point. 

\item Let $\Delta_{01b}\subset \BM_3$ be the closure of the locus parameterizing an elliptic nodal curve attached 
to an elliptic curve at one point. 

\item Let $\Delta_{11}\subset \BM_3$ be the closure of 
the locus parameterizing a chain of three elliptic curves. 

\item Let $\Xi_0\subset \BM_3$ be the closure of the locus parameterizing 
irreducible nodal curves in which the normalization of the node consists of two conjugate points under the hyperelliptic involution. 

\item Let $\Xi_1\subset \BM_3$ be the closure of the locus parameterizing 
two elliptic curves attached at two points. 

\item Let $H_1\subset \BM_3$ be the closure of the locus of curves consisting of an elliptic tail attached to a genus two curve at a Weierstrass point. 
\end{itemize}

The codimension two boundary strata of
$\BM_3$ consist of $\Delta_{00}$, $\Delta_{01a}$, $\Delta_{01b}$, $\Delta_{11}$ and $\Xi_1$. 
We denote the cycle class of a locus by the corresponding small letter, such as $\delta_{00}$ for the class of $\Delta_{00}$. By \cite{FaberChow}, the Chow group $A^2(\BM_3)$ is isomorphic to $N^2(\BM_3)$, with a basis given by $\delta_{00}$, $\delta_{01a}$, $\delta_{01b}$, $\delta_{11}$, $\xi_0$, $\xi_1$ and $h_1$ over $\bbR$.  

The interior of $\Delta_0$ parameterizing irreducible curves with exactly one node is given by
$$\Int \Delta_{0} = \Delta_0 - \Delta_{00} - \Delta_{01a} - \Delta_{01b} - \Xi_1. $$
Faber shows that $A^1(\Int \Delta_0)$ is generated by $\xi_0$ (\cite[Lemma 1.12]{FaberChow}). 
In particular, $A^1(\Delta_0)$ is generated by  $\delta_{00}$, $\delta_{01a}$, $\delta_{01b}$,  $\xi_0$ and $\xi_1$, and 
$A^1(\Delta_0) \to N^2(\BM_3)$ is injective.  

The interior of $\Delta_1$ parameterizing the union of a smooth genus one curve and a smooth genus two curve attached at one point, is given by 
$$ \Int \Delta_1 = \Delta_1 - \Delta_{01a} - \Delta_{01b} - \Delta_{11}. $$
Faber shows that $A^1(\Int \Delta_1)$ is generated by $h_1$ (\cite[Lemma 1.11]{FaberChow}). In particular, 
$A^1(\Delta_1)$ is generated by $\delta_{01a}$, $\delta_{01b}$, $\delta_{11}$ and $h_1$, and 
$A^1(\Delta_1) \to N^2(\BM_3)$ is injective.

Before studying effective cycles of codimension two in $\BM_3$, we need to study effective divisors in $\Delta_0$ and $\Delta_1$. 
Recall the gluing maps 
$$\alpha_{0}: \HD_0 = \TMM_{2,2} \to \Delta_0, $$
$$\alpha_1: \HD_1 = \BM_{1,1}\times \BM_{2,1}\to \Delta_1. $$ 
By the divisor theory of $\BM_{g,n}$ and $\TMM_{g,n}$, we know that $A^1( \HD_0) = N^1( \HD_0)$ is generated by  
the classes of inverse images of $\Delta_{00}$, $\Delta_{01a}$, $\Delta_{01b}$,  $\Xi_0$ and $\Xi_1$ under $\alpha_0$. Similarly, 
$A^1(\HD_1) = N^1(\HD_1)$ is generated by the classes of  inverse images of 
$\Delta_{01a}$, $\Delta_{01b}$, $\Delta_{11}$ and $H_1$ under $\alpha_1$. 
In particular, $(\iota\circ\alpha_i)_{*}: A^1( \HD_i) \to A^2(\BM_3)$ 
is injective for $i=0, 1$, where $\iota: \Delta\hookrightarrow \BM_3$ 
is the inclusion. For ease of notation, we denote the inverse image of a cycle class under $\alpha_i$ by the same symbol. 

\begin{lemma}
\label{lem:effd1}
The classes $\delta_{01a}$, $\delta_{01b}$, $\delta_{11}$ and $h_1$ are extremal in $\Eff^1(\HD_1)$. 
\end{lemma} 

\begin{proof}
The effective cone of divisors on $\BM_{2,1}$ is computed in \cite[Section 3.3]{RullaThesis}. 
Note that $\delta_{01a}$ corresponds to the fiber class of $\HD_1$ over $\BM_{1,1}$, while $\delta_{01b}$, $\delta_{11}$ and $h_1$ are the pullbacks of the extremal divisor classes $\delta_0$, $\delta_1$ and $w$ from $\BM_{2,1}$, respectively, where $w$ 
 is the divisor class of the locus of  curves with a marked Weierstrass point in $ \BM_{2,1}$.  
By Corollary \ref{cor:prod}, $\delta_{01a}$, $\delta_{01b}$, $\delta_{11}$ and $h_1$ are extremal. 
\end{proof}

Let $BN^1_{2}\subset \TMM_{2,2}$ be the closure of the locus of curves such that the two unordered marked points are conjugate under 
the hyperelliptic involution. It has divisor class    
$$ [BN^1_{2}] = -\frac{1}{2}\lambda + \frac{1}{2}\psi - \frac{3}{2} \delta_{0; 2} - \delta_{1; 0}, $$
see e.g. \cite{LoganKodaira}. 

\begin{lemma}
\label{lem:effd0}
The classes $\delta_{0}$, $\delta_{0; 2}$, $\delta_{1; 0}$, $\delta_{1; 1}$ and $[BN^1_{2}]$ are extremal in $\Eff^1(\HD_0)$. Consequently $\delta_{00}$, $\delta_{01a}$, $\delta_{01b}$, $\xi_1$ and $\xi_0$ are extremal as classes in $\Eff^1(\HD_0)$. 
\end{lemma}

\begin{proof}
By Proposition~\ref{prop:boundary-extremal}, the boundary divisors are known to be extremal. To obtain a moving curve $B$ in $BN^1_2$ that has negative intersection with it, fix a general genus two curve $C$ and vary a pair of conjugate points $(p_1, p_2)$ in $C$. Since $[B]\cdot \lambda = 0$, $[B]\cdot \psi = 16$, $[B]\cdot \delta_{0; 2} = 6$ and $[B]\cdot \delta_{1; 0} = 0$, 
we have $[B]\cdot [BN^1_{2}] < 0$.  Hence, $BN^1_{2}$ is extremal. Note that the classes $\delta_{00}$, $\delta_{01a}$, $\delta_{01b}$, $\xi_1$ and $\xi_0$ in $\HD_0$ correspond to 
 $\delta_{0}$, $\delta_{0; 2}$, $\delta_{1; 0}$, $\delta_{1; 1}$ and $[BN^1_{2}]$ in $\TMM_{2,2}$, respectively, thus proving the claim. 
 \end{proof}

\begin{theorem}
\label{thm:eff3}
The classes $\delta_{00}$, $\delta_{01a}$, $\delta_{01b}$, $\xi_1$, $\xi_0$, $\delta_{11}$ and $h_1$ are extremal in $\Eff^2(\BM_3)$.  
\end{theorem}

\begin{proof}
Recall the Torelli map $\tau: \BM_g\to \calA^{\sat}_g$ and the first divisorial contraction $\ps: \BM_g \to \BM_g^{\ps}$ discussed 
in Section~\ref{sec:prelim-moduli}. For $g=3$, note that $\Delta_{01b}$, $\Delta_{11}$ and $H_1$ are contained in $\Delta_1$, and $e_{\ps}(\Delta_{01b}), e_{\ps}(\Delta_{11}), e_{\ps}(H_1) > 0$. Moreover, by Lemma~\ref{lem:effd1} their classes are extremal in $\Eff^1(\BM_{1,1}\times \BM_{2,1})$. Proposition~\ref{prop:subvariety} implies that they are extremal in $\Eff^2(\BM_3)$. 

The strata $\Delta_{00}$ and $\Xi_1$ are contained in $\Delta_0$ and have index $e_{\tau} \geq 2$. If $Z$ is a subvariety of codimension two in $\BM_3$ with $e_{\tau}(Z)\geq 2$, then $Z$ is either contained in $\Delta_0$ or in $\Delta_1$. Since $e_{\tau}(\Delta_1)=1$, if $Z$ is contained in $\Delta_1$ but not in $\Delta_0$, then $Z$ has to be $\Delta_{11}$. By Lemma~\ref{lem:effd0}, $\delta_{00}$ and $\xi_1$ are extremal in $\Eff^1(\TMM_{2,2})$. 
Moreover, $\delta_{00}$, $\xi_1$ and $\delta_{11}$ are linearly independent in $N^2(\BM_3)$. We thus conclude that 
$\delta_{00}$ and $\xi_1$ are extremal in $\Eff^2(\BM_3)$ by Proposition~\ref{prop:subvariety}. 

The remaining cases are $\Delta_{01a}$ and $\Xi_0$. A rational nodal tail does not have moduli.  Consequently, $e_{\tau}(\Delta_{01a}) = 1$ and $e_{\ps}(\Delta_{01a}) = 0$. Similarly we have $e_{\tau}(\Xi_0) = 1$ and $e_{\ps}(\Xi_0) = 0$. The dimension drops are too small to directly apply Proposition~\ref{prop:extremal}. Nevertheless, their $e_{\tau}$ indices are positive, hence we may apply Proposition~\ref{prop:union} to $Y= \Delta_0 \cup \Delta_1$. 
Recall that $N^1(\HD_i) \cong A^1(\HD_i) \to N^2(\BM_3)$ is injective for $i=0, 1$. Moreover, $\Delta_0\cap \Delta_1 = \Delta_{01a} \cup \Delta_{01b}$ and both components are extremal in $\HD_0$ and in $\HD_1$ by Lemmas~\ref{lem:effd1} and~\ref{lem:effd0}. Finally, $\xi_0$ is extremal in $\HD_0$ 
by Lemma~\ref{lem:effd0}. By Proposition~\ref{prop:union}, we  conclude that $\Delta_{01a}$ and $\Xi_0$ are extremal in $\Eff^2(\BM_3)$. 
\end{proof}

\begin{remark}
It would be interesting to find an extremal cycle of codimension two that is not contained in the boundary of $\BM_3$. Some natural geometric non-boundary cycles fail to be extremal. For instance, 
the closure $B_3$ of the locus of bielliptic curves in $\BM_3$ is not extremal. The class of $B_3$ was calculated in terms of 
another basis of $A^2(\BM_3)$ (\cite{FaberPagani}): 
$$ [B_3] = \frac{2673}{2}\lambda^2 - 267 \lambda\delta_0 - 651 \lambda\delta_1 + \frac{27}{2}\delta_0^2 + 69 \delta_0\delta_1 + 
\frac{177}{2}\delta_1^2 - \frac{9}{2}\kappa_2. $$
By \cite[Theorem 2.10]{FaberChow}, we can rewrite it as follows: 
$$ [B_3] = \frac{1}{2}\delta_{00} + \frac{25}{8}\delta_{01a} + \frac{19}{4}\delta_{01b} + 18\delta_{11} + \frac{15}{8}\xi_0 + 15 \xi_1 
+ 10h_1. $$ 
Therefore, $B_3$ is not extremal in $\Eff^2(\BM_3)$. 
\end{remark}

However, if we consider $\BM_{3,1}$ instead,  we are able to find an extremal non-boundary codimension two cycle. First, let us introduce some notation: 
\begin{itemize}
\item Let $H\subset \BM_3$ be the closure of the locus of hyperelliptic curves. 

\item Let $HP\subset \BM_{3,1}$ be the closure of the locus of hyperelliptic curves with a marked point. 

\item Let $HW \subset \BM_{3,1}$ be the closure of the locus of hyperelliptic curves with a marked Weierstrass point. 
\end{itemize}
Clearly $HW$ is a subvariety of codimension two in $\BM_{3,1}$, which is not contained in the boundary. Moreover, $HW\subset HP = \pi^{-1}(H)$ and $\pi: HW \to H$ is generically finite of degree eight, where $\pi: \BM_{3,1} \to \BM_3$ is the morphism forgetting the marked point. 

We will show that $HW$ is an extremal cycle in $\BM_{3,1}$. In order to prove this, we need to understand the divisor theory of $HP$ as well as the boundary components $\Delta_{1; \{1\}}$ and $\Delta_{1; \emptyset}$ of $\BM_{3,1}$. We further introduce the following cycles:  

\begin{itemize}
\item Let $HP_0 \subset HP$ be the closure of the locus of irreducible nodal hyperelliptic curves with a marked point. 

\item Let $HP_1 = HP\cap\Delta_{1; \{1\}} $ be the closure of the locus of curves consisting of a marked genus one component attached to a genus two component at a Weierstrass point. 

\item Let $HP_2 = HP\cap \Delta_{1;\emptyset}$ be the closure of the locus of curves consisting of a genus one component attached at a Weierstrass point to a marked genus two component.  

\item Let $\Theta \subset HP$ be the closure of the locus of two genus one curves attached at two points, one of the curves marked.  

\item Let $D_{01a}\subset  \Delta_{1; \{1\}}$ be the closure of the locus of a marked genus one curve attached to an irreducible nodal curve of geometric 
genus one. 

\item Let $D_{02b}\subset \Delta_{1; \{1\}}$ be the closure of the locus of a marked rational nodal curve attached to a genus two curve. 

\item Let $D_{12} \subset \Delta_{1; \{1\}}\cap \Delta_{1;\emptyset}$ be the closure of the locus of a chain of three curves of genera $2$, $0$ and $1$, respectively, such that the rational component contains the marked point. 

\item Let $D_{11a} \subset \Delta_{1; \{1\}}\cap \Delta_{1;\emptyset}$ be the closure of the locus of a chain of three genus one curves such that one of the two tails contains the marked point. 

\item Let $W_1 \subset \Delta_{1;\emptyset}$ be the closure of the locus of a genus one curve attached to a marked genus two curve such that the marked point is a Weierstrass point. 

\item Let $D_{02a} \subset \Delta_{1;\emptyset}$ be the closure of the locus of a rational nodal curve attached to a marked genus two curve. 

\item Let $D_{01b} \subset \Delta_{1;\emptyset}$ be the closure of the locus of a genus one curve attached to a marked irreducible nodal curve of geometric genus one. 

\item Let $D_{11b} \subset  \Delta_{1;\emptyset}$ be the closure of the locus of a chain of three genus one curves such that the middle component contains the marked point. 

\end{itemize}

Recall the gluing morphisms 
$$\alpha_{1;\{1\}}: \HD_{1;\{ 1\}} = \BM_{1,2} \times \BM_{2,1} \to \Delta_{1; \{1\}}, $$
$$ \alpha_{1; \emptyset}: \HD_{1; \emptyset} = \BM_{1,1}\times \BM_{2,2} \to \Delta_{1; \emptyset}. $$
For ease of notation, we denote  the inverse image of a class under the gluing maps by the same symbol. 

\begin{lemma}
\label{lem:hwhp}
\begin{enumerate}[label={\upshape(\roman*)}]

\item $N^1(HP) \cong A^1(HP)$ is generated by $hw$, $hp_0$, $hp_1$, $hp_2$ and $\theta$.  

\item The divisor classes $hw$, $hp_0$, $hp_1$, $hp_2$ and $\theta$ are extremal in $HP$.  Moreover, 
if a linear combination $a_1 \cdot hp_1 + a_2 \cdot hp_2$ is effective, then $a_1, a_2 \geq 0$.  

\item $N^1(\HD_{1;\{ 1\}} ) \cong A^1(\HD_{1;\{ 1\}} )$ is generated by $hw$, $d_{01a}$, $d_{02b}$, $d_{11a}$ and $d_{12}$.  Moreover, 
if a linear combination $b_1 \cdot hp_1 + b_2 \cdot d_{11a} + b_{3}\cdot d_{12}$ is effective, then $b_i \geq 0$ for all $i$.   
 
\item $N^1(\HD_{1;\emptyset} ) \cong A^1(\HD_{1;\emptyset} )$ is generated by $hp_2$, $w_1$, $d_{01b}$, $d_{02a}$, $d_{11a}$, $d_{11b}$ and $d_{12}$.   Moreover, 
if a linear combination $c_1 \cdot hp_2 + c_2 \cdot d_{11a} + c_3\cdot d_{12}$ is effective, then $c_i \geq 0$ for all $i$.
 
\item The kernel of $N^1(HP)\oplus N^1(\HD_{1;\{ 1\}} ) \oplus N^1(\HD_{1;\emptyset}) \to N^2(\BM_{3,1})$ is generated by 
$$(hp_1, -hp_1, 0), (hp_2, 0, -hp_2), (0, d_{11a}, - d_{11a}) \ \mbox{and} \ (0, d_{12}, - d_{12}).$$ 
\end{enumerate}
\end{lemma}

\begin{proof}
The first two claims essentially follow from \cite[9.2, 10.1]{RullaM0n}. A curve parameterized in $HP$ is a genus three hyperelliptic curve $C$ with a marked point $p$. Hence, it can be identified with an admissible double cover (\cite[3.G]{HarrisMorrison}) of a stable rational curve $R$ branched at eight unordered points $q_1, \ldots, q_8$ with a distinguished marked point $q$ as the image of $p$. Marking the conjugate point $p'$ of $p$ gives $(C, p')$, which is isomorphic to $(C, p)$ under the hyperelliptic involution. Therefore, the data $(R, q_1, \ldots, q_8, q)$ uniquely determines $(C, p)$. The curve $(C,p)$ is not stable if and only if $R$ has a rational tail marked by exactly two of the unordered points or by an unordered point and $q$ or has a rational bridge marked only by $q$. In the first case, the double cover has an unmarked rational bridge. In the second case, the double cover has a rational tail marked by $q$. In the last case, the double cover has a rational bridge with no marked points. In all three cases stabilization contracts the nonstable rational component. Furthermore, when $R$ has a rational tail marked by exactly two unordered points and $q$, the double cover has a rational bridge marked by  $q$. While a rational bridge marked by one point has no moduli, a rational tail with three marked points has one-dimensional moduli. 

Following the notation in \cite{RullaM0n}, let $X_{9,1} = \BM_{0,9}/\mathfrak{S_8}$ be the moduli space of stable genus zero curves with nine marked points $q_1, \ldots, q_8, q$ such that $q_1, \ldots, q_8$ are unordered but $q$ is distinguished. Let $B_k \subset X_{9,1}$ be the boundary divisor 
parameterizing curves with a rational tail marked by $q$ and $k-1$ of the $q_i$ for $2\leq k\leq 7$. There is a morphism $X_{9,1}\to HP$ defined by taking the admissible double cover and stabilizing. This morphism contracts the boundary divisor $B_3$ to the locus of hyperelliptic curves with a marked rational bridge. Numerical equivalence, rational equivalence and linear equivalence over $\bbR$ coincide for divisors on moduli spaces of stable pointed genus zero curves, and the Picard group is generated by boundary divisors. Considering the admissible double covers corresponding to the boundary divisors $B_2$, $B_4$, $B_5$, $B_6$ and $B_7$ of $X_{9,1}$, we see that their images in $HP$ correspond to $HW$, 
$HP_1$, $\Theta$, $HP_2$ and $HP_0$, respectively, hence (i) follows. The same curves proving the extremality of the boundary divisors in this case (see \cite[10.1]{RullaM0n}) also prove (ii).  

For (iii), $hw$ and $d_{ij}$ correspond to the pullbacks of generators of $N^1(\BM_{1,2})\cong A^1(\BM_{1,2})$ and of $N^1(\BM_{2,1})\cong A^1(\BM_{2,1})$, where $\HD_{1;\{1\}} = \BM_{1,2}\times \BM_{2,1}$. 
Moreover, the interior $\MM_{1,2}\times \MM_{2,1}$ is affine. 
Hence, these classes generate $N^1(\HD_{1; \{1 \}})\cong A^1(\HD_{1; \{1 \}})$. 
Suppose that $b_1 \cdot hp_1 + b_2 \cdot d_{11a} + b_{3} d_{12}$ is an effective divisor class. Let $\pi_1$ and $\pi_2$ be the projections of $\HD_{1; \{1 \}}$ to $\BM_{1,2}$ and $\BM_{2,1}$, respectively. 
Then, we have $hp_1 = \pi_{2}^{*}w$, $d_{11a} = \pi_{2}^{*} \delta_{1; \{ 1\}}$ and $d_{12} = \pi_{1}^{*} \delta_{0; \{ 1,2\}}$, where $w$ is the divisor class of Weierstrass points. Hence, $b_3 \geq 0$ and 
$b_1 \cdot w + b_2 \cdot \delta_{1; \{ 1\}}$ is effective in $\BM_{2,1}$. Since $w$ and $\delta_{1; \{ 1\}}$ span a face of $\Eff^1(\BM_{2,1})$ (\cite[Corollary 3.3.2]{RullaThesis}), we conclude that $b_1, b_2 \geq 0$. A similar argument yields (iv). 

For (v), since $HP\cap \Delta_{1;\{1\}} = HP_1$, $HP\cap \Delta_{1;\emptyset} = HP_2$ and 
$\Delta_{1;\{1\}} \cap \Delta_{1;\emptyset} = D_{12}\cup D_{11a}$, 
it suffices to prove that $hw$, $hp_0$, $hp_1$, $hp_2$, $\theta$, $w_1$, 
$d_{01a}$, $d_{01b}$, $d_{02a}$, $d_{02b}$, $d_{11a}$, $d_{11b}$ and $d_{12}$ are independent in $N^2(\BM_{3,1})$. 
Suppose that they satisfy a relation 
\begin{eqnarray}
\label{eq:hwhpdij}
a \cdot hw + t\cdot \theta + s \cdot w_1 + \sum_i b_i \cdot hp_i + \sum_{i,j,k} c_{ijk} \cdot d_{ijk} = 0 \in N^2(\BM_{3,1}) 
\end{eqnarray}
for $a, b_i, c_{ijk}, s, t \in \bbR$. 
The morphism $\pi: \BM_{3,1}\to \BM_3$ contracts these cycles except $HW$, $W_1$ and $D_{12}$, where $\pi(HW) = H$ and $\pi(W_1) = \pi(D_{12}) = \Delta_1$. 
Since $\pi$ is flat and $h, \delta_1$ are independent in $\BM_3$, applying $\pi_{*}$ to \eqref{eq:hwhpdij}
we  conclude that $a = 0$.  
 
For the remaining cycles,  
we have $\pi(HP_1) = \pi(HP_2) = H_1$, $\pi(HP_0) = \Xi_0$, $\pi(\Theta) = \Xi_1$, 
$\pi(D_{02a} ) = \pi(D_{02b}) = \Delta_{01a}$, $\pi(D_{01a}) = \pi(D_{01b}) = \Delta_{01b}$ and $\pi(D_{11a}) = \pi(D_{11b})=\Delta_{11}$. 
The images of $W_1$ and $D_{12}$ are contained in $\Delta_1$. By \cite{FaberChow}, 
the subspace spanned by $\xi_0$ and $\xi_1$ has zero intersection with $A^1(\Delta_1)$ in $N^2(\BM_{3})$. Intersect \eqref{eq:hwhpdij} with an ample divisor class and apply $\pi_{*}$. We thus conclude that $t = b_0 = 0$. 

Relation \eqref{eq:hwhpdij}  reduces to 
\begin{eqnarray}
\label{eq:hwhpdij-1}
s \cdot w_1 +  b_1 \cdot hp_1 + b_2 \cdot hp_2+ \sum_{i,j,k} c_{ijk} \cdot d_{ijk} = 0.  
\end{eqnarray}
Recall the pseudostable map $\ps: \BM_{3,1}\to \BM_{3,1}^{\ps}$. 
Applying $\ps_{*}$, we obtain that 
\begin{eqnarray}
\label{eq:hwhpdij-2}
b_1 \cdot (\ps_{*} hp_1) + c_{01a} \cdot (\ps_{*}d_{01a}) + c_{02a}\cdot (\ps_{*} d_{02a}) + c_{02b}\cdot (\ps_{*} d_{02b}) = 0 
\end{eqnarray} 
in $N^2(\BM_{3,1}^{\ps})$, since the other summands in \eqref{eq:hwhpdij-1} are contracted by $\ps$. Take a general net $S_1$ of plane quartics with a marked base point. Then $S_1$ intersects $\ps(D_{02a})$ at finitely many points parameterizing cuspidal quartics and $S_1$ does not intersect the other cycles in \eqref{eq:hwhpdij-2}, which implies that $c_{02a} = 0$. Let $\pi_{\ps}: \BM_{3,1}^{\ps} \to \BM_{3}^{\ps}$ be the morphism forgetting the marked point and pseudostablizing. Note that $e_{\pi_{\ps}}(\ps(D_{02b})) = 1$ and $e_{\pi_{\ps}}(\ps(HP_1)) = e_{\pi_{\ps}}(\ps(D_{01a})) = 2$. Take an ample divisor class in $\BM_{3,1}^{\ps}$, intersect  \eqref{eq:hwhpdij-2} and apply $\pi_{\ps*}$. We thus conclude that $c_{02b} = 0$. Finally, take a pencil of plane cubics, mark a base point, take another base point and use it to attach the pencil to a varying  point in a fixed genus two curve. We obtain a two-dimensional family $S_2$ in $\BM_{3,1}^{\ps}$ such that 
$[S_2] \cdot (\ps_{*} hp_1)= -6$ and $[S_2] \cdot (\ps_{*} d_{01a}) = 0$, which implies that $b_1 = c_{01a} = 0$.


Relation \eqref{eq:hwhpdij} further reduces to 
\begin{eqnarray}
\label{eq:hwhpdij-3}
s \cdot w_1 + b_2 \cdot hp_2+ c_{12} \cdot d_{12} + c_{11a} \cdot d_{11a} + c_{01b} \cdot d_{01b} + c_{11b} \cdot d_{11b}= 0.  
\end{eqnarray}
Take a general net 
$S_3$ of plane quartics such that $S_3$ contains finitely many general cuspidal curves. Take a line $L$ passing through a Weierstrass point 
of the normalization of one of the cuspidal quartics $C$ and make $L$ general other than that. Mark the intersection points of $L$ with curves in $S_3$, make a base change, perform stable reduction and still denote by $S_3$ the resulting family in $\BM_{3,1}$. Note that $S_3$ does not intersect the summands in \eqref{eq:hwhpdij-3} except $W_1$, and it intersects $W_1$ along a one-dimensional locus parameterizing 
the normalization of $C$ with a pencil of elliptic tails attached at the inverse image of the cusp of $C$, which implies that 
$[S_3]\cdot w_1 < 0$ since the normal bundle of elliptic tails in $\Delta_1$ has negative degree restricted to this family. We thus conclude that $s = 0$. 

Set $s=0$, intersect \eqref{eq:hwhpdij-3} with the $\psi$-class of the marked point and apply $\pi_{*}$. Since the marked point lies in the rational bridge for curves in $D_{12}$, it implies that $\psi\cdot d_{12} = 0$. We thus obtain that 
$$ (3b_2) \cdot h_1 + (c_{11a} + 2 c_{11b} ) \cdot \delta_{11} + c_{01b} \cdot \delta_{01b} = 0 $$
in $N^2(\BM_3)$. Since $h_1, \delta_{11}$ and $\delta_{01b}$ are independent in $\BM_3$ (\cite{FaberChow}), we conclude that $b_2 = c_{01b} = 0$ and $c_{11a} + 2 c_{11b} = 0$. 

Now Relation \eqref{eq:hwhpdij} reduces to  
\begin{eqnarray}
\label{eq:hwhpdij-4}
c_{12} \cdot d_{12} - 2 c_{11b} \cdot d_{11a} + c_{11b} \cdot d_{11b}= 0.  
\end{eqnarray}
Attach two pencils of plane cubics to a smooth genus one curve $E$ at two general points, and mark a third general point in $E$. We obtain a two-dimensional family $S_4$ in $\BM_{3,1}$ such that 
$[S_4] \cdot d_{12} = [S_4]\cdot d_{11a} = 0$ and $[S_4]\cdot d_{11b} = 1$. Plugging in 
\eqref{eq:hwhpdij-4}, we obtain that $c_{11b} = 0$, hence $c_{12} \cdot d_{12} = 0$ and $c_{12} = 0$. 

\end{proof}

Now we are ready to prove that $HW$ is extremal in $\BM_{3,1}$. In fact, all extremal divisors in $HP$ are extremal  as codimension two cycles in $\BM_{3,1}$. First, observe that $\pi: \BM_{3,1} \to \BM_3$ is flat of relative dimension one. As a consequence, if $Z\subset \BM_{3,1}$ is an irreducible subvariety 
of codimension two, then $\pi(Z)\subset \BM_3$ has codimension either one or two, i.e. $e_{\pi}(Z) = 0$ or $1$. 

\begin{theorem}
\label{thm:hw}
The cycle classes of $HW$, $HP_0$, $HP_1$, $HP_2$ and $\Theta$ are extremal in $\Eff^2(\BM_{3,1})$. 
\end{theorem}

\begin{proof}
We prove it for $HW$ first. Suppose that 
\begin{eqnarray}
\label{eqn:hw}
hw= \sum_{i=1}^r a_i [Y_i] + \sum_{j=1}^s b_j [Z_j] \in N^2(\BM_{3,1})
\end{eqnarray}
where $a_i, b_j > 0$ and $Y_i, Z_j\subset \BM_{3,1}$ are irreducible subvarieties of codimension two such that $e_\pi (Y_i) = 0$ and $e_{\pi} (Z_j) = 1$. 
Since $\pi_{*}[Z_j] = 0$, we have 
$$ 8 \cdot h = \sum_{i=1}^r a_i \cdot \pi_{*}[Y_i] \in N^1(\BM_3). $$
Since $H$ is an extremal and rigid divisor in $\BM_3$ (see e.g. \cite[2.2]{RullaThesis}), it follows that $\pi(Y_i) = H$. Hence, $Y_i\subset \pi^{-1}(H) = HP$ for all $i$.

Next, we show that $Z_j$ is contained in $HP\cup \Delta_{1;\{1\}}\cup \Delta_{1;\emptyset}$ for all $j$. Let $S$ be a general net of plane quartic curves with a marked base point. By taking $S$ general enough, we can assume that $S$ avoids any phenomenon of codimension three or higher. In particular, we may assume that $S$ does not contain any reducible or nonreduced elements and $S$ has finitely many cuspidal elements such that the base point is not any of the cusps. 
Furthermore, we may specify $S$ to contain a specific smooth or non-hyperelliptic irreducible nodal quartic in $Z_1$ while preserving these properties.  Then $S$ induces a surface in $\BM_{3,1}$ such that $[S] \cdot hw =0$, $[S] \cdot [Y_i] \geq 0$, $[S]\cdot [Z_j] \geq 0$ for $j\neq 1$ and 
$[S]\cdot [Z_1] > 0$. Intersecting both sides of \eqref{eqn:hw} with $S$ leads to a contradiction. We thus conclude that $Y_i, Z_j \subset HP\cup \Delta_{1;\{1\}}\cup \Delta_{1;\emptyset}$ for all $i, j$. By Lemma~\ref{lem:hwhp}, we can apply Proposition~\ref{prop:union} to $Y = HP \cup \Delta_{1;\{1\}}\cup \Delta_{1;\emptyset}$, thus proving that $hw$ is extremal in $\Eff^2(\BM_{3,1})$.

For $HP_0$, $HP_1$, $HP_2$ and $\Theta$, note that their $e_{\pi}$ indices are all equal to one and their images under $\pi$ are contained in $H$. If any of them satisfies a relation like~\eqref{eqn:hw}, 
by Proposition~\ref{prop:main-tool} the $Y_i$ terms cannot exist in the effective expression. Then the same argument in the previous paragraph implies that 
the remaining terms $Z_j$ are contained in $ HP\cup \Delta_{1;\{1\}}\cup \Delta_{1;\emptyset}$ for all $j$. Hence, we can apply Lemma~\ref{lem:hwhp} and Proposition~\ref{prop:union}  to conclude that they are extremal in $\Eff^2(\BM_{3,1})$. 
\end{proof}

\section{The effective cones of $\BM_{g}$ for $g\geq 4$}
\label{sec:genus-g}

In this section, we study effective cycles of codimension two on $\BM_g$ for $g\geq 4$. The method is similar to the case $g=3$. 
There is a topological stratification of $\BM_g$, where the strata are indexed by dual graphs of stable nodal curves. The codimension two boundary strata consist of curves with at least two nodes. We recall the notation introduced in \cite{EdidinChow} for these strata: 
\begin{itemize}
\item Let $\Delta_{00}$ be the closure of the locus in $\BM_g$ parameterizing irreducible curves with two nodes. 
\item For $1\leq i\leq j\leq g-2$ and $i+j\leq g-1$, let $\Delta_{ij}$ be the closure of the locus in $\BM_g$ parameterizing a chain of three curves of genus $i$, $g-i-j$ and $j$, respectively. 
\item For $1\leq j\leq g-1$, let $\Delta_{0j}$ be the closure of the locus in $\BM_g$ parameterizing a union of a genus $j$ curve and an irreducible nodal curve of geometric genus $g-1-j$, attached at one point. 
\item For $1\leq i \leq [(g-1)/2]$, let $\Theta_i$ be the closure of the locus in $\BM_g$ parameterizing a union of a curve of genus $i$ and a curve of genus $g-i-1$, attached at two points. 
\end{itemize}

In order to study higher codimension cycles on $\BM_g$, we need to understand the divisor theory of the boundary components. 
Recall the gluing morphisms $\alpha_{i}: \HD_i \to \Delta_i$ for $0\leq i\leq [g/2]$. By \cite[p. 69]{FaberCodim} and \cite[Section 4]{EdidinChow}, $A^1(\Delta) \to N^2(\BM_g)$ is injective. Moreover, 
$(\iota\circ \alpha_i)_{*}: N^1(\HD_i)\cong A^1(\HD_i) \to N^2(\BM_g)$
is injective, where $\iota: \Delta \hookrightarrow \BM_g$ is the inclusion. 
As before, we denote the class of a locus in $\Delta_i$ and in $\HD_i$ by the same symbol. Whenever we use $g/2$ as an index, the corresponding term exists if and only if $g$ is even.  

\begin{lemma}
\label{lem:delta1} 
\begin{enumerate}[label={\upshape(\roman*)}]
\item The classes $\delta_{11}$ and $\delta_{1g-2}$ are extremal in $\Eff^1(\HD_1)$. 

\item The class $\delta_{0g-1}$ is extremal in $\Eff^1(\HD_0)$.
\end{enumerate}
\end{lemma}

\begin{proof}
Note that $\HD_1 = \BM_{1,1}\times \BM_{g-1, 1}$. The classes $\delta_{11}$ and $\delta_{1g-2}$ are the pullbacks of 
boundary classes $\delta_{1;\empty}$ and $\delta_{g-2;\empty}$ from $\BM_{g-1,1}$, respectively. Now (i) follows from 
Corollary~\ref{cor:prod} and Proposition~\ref{prop:boundary-extremal}. 


For (ii), $\delta_{0g-1}$ corresponds to the boundary class $\delta_{0; 2}$ in $\HD_0 = \TMM_{g-1,2}$. Hence, the claim follows from 
Proposition~\ref{prop:boundary-extremal}. 
\end{proof}

\begin{lemma}
\label{lem:i1} 
\begin{enumerate}[label={\upshape(\roman*)}]
\item The linear combination 
$ a_0 \delta_0 + \sum_{i=1}^{g-1}a_i \delta_{i,1} $ is  effective on $\BM_{g,1}$ if and only if 
all the coefficients are nonnegative. 

\item For $0 < i < g/2$, the linear combination $$\sum_{k=0}^{i-1} (c_{ki}  \delta_{ki} + d_{k g-i} \delta_{k g-i})
+ \sum_{k=i+1}^{[g/2]} ( c_{ik} \delta_{ik} + d_{i g-k} \delta_{i g-k}) $$ 
is  effective on $\HD_i$ if and only if all the coefficients are nonnegative. 

\item For $g$ even and $i = g/2$, the linear combination 
$$ \sum_{k=0}^{g/2-1} b_{k g/2} \delta_{k g/2}$$ 
is effective on $\HD_{g/2}$ if and only if all the coefficients are nonnegative. 
\end{enumerate}
\end{lemma}

\begin{proof}
For (i), take a nodal curve of geometric genus $g-1$ and vary a marked point on it. 
We obtain a curve $C_0$ moving in $\Delta_0$ such that $[C_0]\cdot \delta_0 < 0$ and $[C_0]\cdot \delta_{i,1} = 0$ for all $i$. 
For $1\leq i \leq g-2$, sliding a genus $i$  curve with a marked point along a genus $g-i$ curve, we obtain a curve 
$C_i$ moving in $\Delta_{i,1}$ such that $[C_i] \cdot \delta_{i,1} < 0$, 
$[C_i]\cdot\delta_0 = 0$ and $[C_i] \cdot \delta_{j,1} = 0$ for $j\neq i$. 
Finally, attach a genus $g-1$ curve with a marked point to a genus one curve at a general point and vary the marked point in the component of genus $g-1$. We obtain a 
curve $C_{g-1}$ moving in $\Delta_{g-1,1}$ such that $[C_{g-1}]\cdot \delta_0 = 0$, $[C_{g-1}]\cdot \delta_{1,1} = 1$, 
$[C_{g-1}]\cdot \delta_{g-1,1} = 4-2g < 0$ and $[C_{g-1}]\cdot \delta_{i,1} = 0$ for $2\leq i \leq g-2$. 

Suppose that $D$ is an effective divisor on $\BM_{g,1}$ with class 
$[D] = a_0 \delta_0 + \sum_{i=1}^{g-1}a_i\delta_{i,1}$. 
If the support of $D$ is contained in the union of $\Delta_0$ and the $\Delta_{i,1}$, we are done, because the boundary divisor classes are linearly independent. Suppose that $D$ does not contain any boundary components in its support. Since the curves constructed are moving in the respective boundary components, we have $[D]\cdot [C_i] \geq 0 $ for $0\leq i \leq g-1$.  
We thus conclude that $ a_0 \leq 0$, $a_i \leq 0$ for $1\leq i \leq g-2$ and $(4-2g)a_{g-1} + a_1 \geq 0$, hence 
$a_{g-1}\leq 0$. It follows that 
$[D] + (-a_0)\delta_0 + \sum_{i=1}^{g-1} (-a_i)\delta_{i,1} = 0$, leading to a contradiction, because the class of a positive sum of effective divisors cannot be zero. Indeed we have proved that an effective divisor with class $a_0 \delta_0 + \sum_{i=1}^{g-1}a_i \delta_{i,1}$ has to be supported in the union of 
$\Delta_0$ and the $\Delta_{i,1}$. 

For (ii) and (iii), the boundary classes in the linear combination are the pullbacks of the boundary classes from  $\BM_{i, 1}$ and $\BM_{g-i,1}$, respectively, hence the claims follow from (i). 
\end{proof}

\begin{theorem}
\label{thm:eff2mg}
For $g\geq 2$, every codimension two boundary stratum of $\BM_g$ is extremal in $\Eff^2(\BM_g)$. 
\end{theorem}

\begin{proof}
When $g=2$, the two $F$-curves $\Delta_{00}$ and $\Delta_{01}$ are dual to the two nef divisors $\lambda$ and $12 \lambda - \delta$, respectively, and form the extremal rays of the Mori cone of curves. The case $g=3$ is covered by Theorem~\ref{thm:eff3}. Hence,  we may assume that $g\geq 4$. 

Recall that the extended Torelli map $\tau: \BM_g\to \calA^{\sat}_g$ sends a stable curve to the product of the Jacobians of the irreducible components of its normalization. Suppose that $Z$ is a codimension two boundary stratum of $\BM_g$ such that a general curve $C$ parameterized by $Z$ does not have an irreducible component of geometric genus less than two. This assumption ensures that selecting two points (as the inverse image of a node) in the normalization of $C$ has two-dimensional moduli. We thus conclude that 
$$e_\tau(Z) = \dim Z - \dim \tau(Z) = 2 + 2 = 4, $$
because $\tau (C_1) = \tau (C_2)$ if $C_1$ and $C_2$ have the same normalization. If a general curve $C$ parameterized by $Z$ 
contains exactly one irreducible component $E$ of geometric genus one and no other components of geometric genus less than two, then $e_\tau(Z) = 3$ since choosing $n$ points in the normalization of $E$ has $(n-1)$-dimensional moduli. 

Conversely, if $Z'\subset \BM_g$ is a subvariety of codimension two such that 
$e_\tau(Z') \geq 3$, then every curve parameterized by $Z'$ has at least two nodes. Therefore, $Z'$ has to be one of the codimension two boundary strata. Since the classes of the codimension two boundary strata of $\BM_g$ are independent (\cite[Theorem 4.1]{EdidinChow}), we conclude that $[Z]$ is extremal in $\Eff^2(\BM_g)$ by Proposition~\ref{prop:extremal}. 

The remaining boundary strata are: $\Delta_{11}$ whose general point parameterizes a chain of three curves of genera $1$, $g-2$ and $1$; $\Delta_{1g-2}$ whose general point parameterizes a chain of three curves of genera $1$, $1$ and $g-2$; $\Delta_{0g-1}$ whose general point parameterizes a curve with a rational nodal tail. They are all contained in the boundary divisor $\Delta_1$. 

Recall that $\ps: \BM_g \to \BM_g^{\ps}$ contracts the boundary divisor $\Delta_1$ by replacing elliptic tails by cusps. A subvariety of $\BM_g$ contracted by $\ps$ has to be contained in $\Delta_1$. Note that 
$e_{\ps}(\Delta_{11}), e_{\ps}(\Delta_{1g-2}) > 0$, and by Lemma~\ref{lem:delta1} (i), $\delta_{11}$ and $\delta_{1g-2}$ are extremal in $\Eff^1(\HD_1)$. Moreover, $N^1(\HD_1) \cong A^1(\HD_1)\to N^2(\BM_g)$ is injective. Therefore, we conclude that $\delta_{11}$ and $\delta_{1g-2}$ are extremal in $\Eff^2(\BM_g)$ by Proposition~\ref{prop:subvariety}. 

Finally, for $\Delta_{0g-1}$, the argument is similar to that of $\Delta_{01a}$ in the proof of Theorem~\ref{thm:eff3}. Since $e_\tau(\Delta_{0g-1}) > 0$, we 
can apply Proposition~\ref{prop:union} to the setting $Y = \Delta_0 \cup \cdots \cup \Delta_{[g/2]}$ and $D_0 = \Delta_0$. Note that 
$\Delta_i\cap \Delta_j = \Delta_{ij} \cup \Delta_{i g-j}$ and $\Delta_{i}\cap \Delta_{g/2} = \Delta_{i g/2}$ for $0 \leq i < g/2$. Lemmas~\ref{lem:delta1} (ii) and~\ref{lem:i1} ensure that Proposition~\ref{prop:union} applies in this case, thus proving that $\Delta_{0g-1}$ is extremal in $\Eff^2(\BM_g)$. 
\end{proof}

The techniques of Theorem \ref{thm:eff2mg} allow us to show that  certain boundary strata of arbitrarily high codimension are extremal. In the next theorem, we will give the simplest examples. Given a stable dual graph $\Gamma$ of arithmetic genus $g$, let $\Delta_{\Gamma}$ denote the closure of the stratum indexed by $\Gamma$ in the topological stratification of $\BM_g$.  Denote the class of $\Delta_{\Gamma}$ by $\delta_{\Gamma}$.  The codimension of $\Delta_{\Gamma}$ is the number of edges in $\Gamma$.

Let $\kappa(g,r)$ denote the set of stable dual graphs of arithmetic genus $g$ with $r$ edges such that the geometric genera of the curves associated to each node is at least two. Let $\kappa'(g,r)$ be the set of stable dual graphs of arithmetic genus $g$ with $r$ edges such that the geometric genera of the curves associated to all the nodes but one is at least two and the remaining genus is at least one. Clearly $\kappa (g,r) \subset \kappa'(g,r)$.  With this notation, we have the following application of Propositions \ref{prop:main-tool} and \ref{prop:extremal}.

\begin{lemma}\label{lemma:high}
Let $\Gamma \in \kappa(g,r)$ (resp. $\Gamma \in \kappa'(g,r)$). If  the class $\delta_{\Gamma}$ is not in the span of the classes $\delta_{\Xi}$ for $\Gamma \not= \Xi \in \kappa(g,r)$ (resp. $\Gamma \not= \Xi \in \kappa'(g,r)$), then $\delta_{\Gamma}$ is an extremal cycle of codimension $r$ in $\BM_g$.
\end{lemma}

\begin{proof}
If $\Gamma \in \kappa(g,r)$, then $e_\tau(\Delta_{\Gamma}) = 2r$. If $\Gamma \in \kappa'(g,r) - \kappa(g,r)$, then $e_\tau(\Delta_{\Gamma}) = 2r-1$. Conversely, if $Z \subset \BM_{g}$ is an irreducible variety of codimension $r$ such that $e_\tau(Z) \geq 2r$, then $Z= \Delta_{\Gamma}$ for some $\Gamma \in \kappa(g,r)$. Similarly, if $e_\tau(Z) \geq 2r-1$, then $Z = \Delta_{\Gamma}$ for some $\Gamma \in \kappa'(g,r)$. By Proposition \ref{prop:main-tool}, any effective linear combination expressing $\delta_{\Gamma}$ for $\Gamma \in \kappa(g,r)$ (resp. $\kappa'(g,r)$) must be of the form $\sum a_{\Xi} \delta_{\Xi}$ with $\Xi \in \kappa(g,r) $ (resp. $\kappa'(g,r)$).  Since $\delta_{\Gamma}$ is not in the span of $\delta_{\Xi}$ for $\Gamma \not= \Xi \in \kappa(g,r)$ (resp. $\kappa'(g,r)$), all the coefficients except for $a_{\Gamma}$ have to be zero. Therefore, $\delta_{\Gamma}$ is an extremal cycle of codimension $r$. 
\end{proof}

In view of Lemma \ref{lemma:high}, it is interesting to determine when the classes $\delta_{\Gamma}$ are independent. Using test families, we give examples of independent classes, which by Lemma \ref{lemma:high} are extremal. 

Let $T_r(g) \subset \kappa(g,r)$ be the set of dual graphs whose underlying abstract graph is a tree with $r$ leaves. If $\Gamma \in T_r(g)$, then the general point of $\Delta_{\Gamma}$ parameterizes curves with $r$ components $C_1, \dots, C_r$ each attached at one point to distinct points on a curve $C_0$. In addition, each of the curves $C_i$ for $0 \leq i \leq r$ have genus at least 2. Let $T_r'(g) \subset \kappa'(g,r)$ be the set of dual graphs whose underlying abstract graph is a tree with $r$ leaves and the central component $C_0$ has genus one. Let $T_r''(g) \subset \kappa'(g,r)$ be the set of dual graphs whose underlying abstract graph is a tree with $r$ leaves, one of the components $C_i$, $i\not=0$, has genus one and the genus $g_0$ of the central component satisfies $4g_0 + 9 \geq r$.  

Let $L_{r-1, g_0}(g) \subset \kappa(g,r)$ denote the set of dual graphs whose underlying abstract graph consists of a vertex  $v_0$ with a self-loop and $r-1$ leaves, where $v_0$ is assigned a curve of geometric genus $g_0$ satisfying $4g_0 + 9 \geq r$. If $\Gamma \in L_{r-1, g_0}(g)$, then the general point of $\Delta_{\Gamma}$ parameterizes curves with $r-1$ components $C_1, \dots, C_{r-1}$ attached at one point to distinct smooth points on a one-nodal central curve $C_0$ (corresponding to $v_0$)  of geometric genus $g_0$.

\begin{theorem}\label{thm:arbitrarily-high}
If $\Gamma$ is a dual graph in $T_r(g) \cup T_r'(g) \cup T_r''(g) \cup L_{r-1, g_0}(g)$, then $\delta_{\Gamma}$ is independent from all the classes $\delta_{\Xi}$ for $\Gamma \not= \Xi \in \kappa'(g,r)$ and $\delta_{\Gamma}$ is an extremal codimension $r$ class in $\BM_g$.
\end{theorem}

\begin{proof}
Suppose there exists a linear relation 
\begin{equation}\label{lin-rel}
\sum_{\Xi \in \kappa'(g,r)} c_{\Xi} \delta_{\Xi}=0
\end{equation} 
among the classes with $\Xi \in \kappa'(g,r)$. Using test families, we will show that if $\Gamma \in T_r(g) \cup T_r'(g) \cup T_r''(g) \cup L_{r-1, g_0}(g)$, then the coefficient $c_{\Gamma}$ in the linear relation has to be zero proving that the class $\delta_{\Gamma}$ is independent from the classes indexed by $\kappa'(g,r)$.

First, assume that $\Gamma \in T_r(g) \cup T_r'(g)$. Fix a general curve in $\Delta_{\Gamma}$. Recall that the curve has $r$ components $C_1, \dots, C_r$ attached to a central component $C_0$. Let $X$ be the $r$-dimensional test family obtained by varying the attachment point on the curves $C_i$ for $1 \leq i \leq r$. Every curve in the family $X$ has the same dual graph $\Gamma$. Hence, $[X] \cdot \delta_{\Xi}=0$ for every $\Gamma \not= \Xi \in \kappa'(g,r)$.  By \cite[Lemma 3.5]{EdidinChow}, the intersection number $[X] \cdot \delta_{\Gamma} = \prod_{i=1}^r (2 - 2 g(C_i))$. Since $g(C_i) > 1$, $[X] \cdot \delta_{\Gamma} \not= 0$.
Intersecting  Relation \eqref{lin-rel}, we conclude that $c_{\Gamma} =0$. Hence, $\delta_{\Gamma}$ is independent of the classes $\delta_{\Xi}$ with $\Gamma \not= \Xi \in \kappa'(g,r)$ and  by Lemma \ref{lemma:high} is an extremal codimension $r$ class. 

Next,  suppose $\Gamma \in L_{r-1, g_0}(g)$. A general point of $\Delta_{\Gamma}$ has $r-1$ curves $C_1, \dots, C_{r-1}$ attached to a one-nodal curve $C_0$ of geometric genus $g_0$. Let $P$ be a general pencil of curves of type $(2, g_0 +2)$ on $\PP^1 \times \PP^1$. Such a pencil has $4g_0 + 8$ base points. Let $Y$ be the $r$-dimensional family obtained by varying $C_0$ in the pencil $P$ and attaching $C_1, \dots, C_{r-1}$ at distinct base-points of $P$ along varying points on $C_1, \dots, C_{r-1}$.  We need the inequality $4g_0 + 9 \geq r$ to ensure that we can form this family. The family $Y$ intersects the codimension $r$ boundary components exactly when a member of the pencil becomes nodal. Hence, $Y$ intersects the codimension $r$  boundary stratum $\Delta_{\Gamma}$ and is disjoint from all other codimension $r$ bounday strata. We conclude that $[Y] \cdot \delta_{\Xi} = 0$ for $\Gamma \not= \Xi \in \kappa'(g,r)$. By \cite[Lemma 3.4]{EdidinChow}, $[Y] \cdot \delta_{\Gamma} = (8g_0 + 12) \prod_{i=1}^{r-1} (1 - 2g(C_i))$, which is not zero since $g(C_i) \geq 2$ for $0 \leq i \leq r-1$. Intersecting Relation \eqref{lin-rel} with $Y$, we conclude that $c_{\Gamma} =0$. Hence, $\delta_{\Gamma}$ is independent of the classes $\delta_{\Xi}$ with $\Gamma \not= \Xi \in \kappa'(g,r)$ and  by Lemma \ref{lemma:high} is an extremal codimension $r$ class. 

Finally, suppose that $\Gamma \in T_r''(g)$. A general point of $\Delta_{\Gamma}$ parameterizes a curve with $r$ components $C_1, \dots, C_r$ attached to a central component $C_0$, where $4 g(C_0) + 9 \geq r$. In addition, one of the components, say $C_1$, has genus one. Let $Q$ be a general pencil of plane cubics. Let $Z$ be the $r$-dimensional family obtained by attaching $C_0$ to the pencil $Q$ at a base point and varying the points of attachments on $C_2, \dots, C_r$. The pencil $Q$ has $12$ nodal members. Since the geometric genus of the nodal curves in $Q$ are zero,  the only  boundary strata $\delta_{\Xi}$, with $\Xi \in \kappa'(g,r)$, that  $Z$ intersects are $\Delta_{\Gamma}$ and $\Delta_{\Psi}$ with $\Psi \in L_{r-1, g(C_0)}(g)$. By \cite[Lemma 3.5]{EdidinChow}, $[Z] \cdot \delta_{\Gamma} = (-1)\cdot \prod_{i=2}^r (2- 2g(C_i)) \not= 0$. By the previous paragraph, the coefficients $c_{\Psi}$ in the Relation \eqref{lin-rel} are zero. Hence, intersecting the relation with $Z$, we conclude that $c_{\Gamma} =0$. Hence, $\delta_{\Gamma}$ is independent of the classes $\delta_{\Xi}$ with $\Gamma \not= \Xi \in \kappa'(g,r)$ and  by Lemma \ref{lemma:high} is an extremal codimension $r$ class. This concludes the proof of the theorem.
\end{proof}

\begin{remark}
If $2r+2 \geq g$, then the sets $T_r(g), T_r'(g)$ and $T_r''(g)$ are non-empty. Hence, Theorem \ref{thm:arbitrarily-high} gives examples of extremal cycles of arbitrarily high codimension. 
\end{remark}

As mentioned before, it would be interesting to find an extremal higher codimension cycle not contained in the boundary of 
$\BM_g$. Since for $g\geq 1$ any birational morphism from $\BM_{g,n}$ has its exceptional locus contained in the 
boundary (\cite[Corollary 0.11]{GibneyKeelMorrison}), one cannot directly apply Proposition~\ref{prop:main-tool}. 
Nevertheless, we are able to find a non-boundary extremal codimension two cycle in $\BM_4$. 

Let $H\subset \BM_4$ be the closure of the locus of hyperelliptic curves and let $GP\subset \BM_4$ be the closure of the Gieseker-Petri special curves, i.e. curves whose canonical embeddings are contained in a quadric cone in $\bbP^3$. 

\begin{theorem}
\label{thm:hyp-4}
The cycle class of $H$ is extremal in $\Eff^2(\BM_4)$. 
\end{theorem}

\begin{proof}
The proof relies on analyzing the canonical model of a genus four curve $C$. If $C$ is $3$-connected and non-hyperelliptic, then the canonical embedding of $C$ is 
a complete intersection of a quadric surface $Q$ and a cubic surface $T$ in $\bbP^3$. If $Q$ is smooth, consider the linear system $V = |\OO(3,3)|$ on $Q$. Let $U\subset V$ be the open locus of smooth curves. Let $NI\subset V$ be the codimension one locus of nodal irreducible curves. 
Let $CP\subset V$ be the codimension two locus of irreducible curves with a cusp. Note that $V\backslash (U\cup NI \cup CP)$ has codimension three 
in $V$. Curves in  $U\cup NI$ are stable and curves in $CP$ are pseudostable. 

If $Q$ is a quadric cone with vertex $v$, let $F_2$ be the Hirzebruch surface obtained by blowing up $v$. Let $E$ be the exceptional $(-2)$-curve with class $e$ and let $f$ be the ruling class of $F_2$. If $[C]\in GP$ is general, then $C$ has class 
$3e + 6f$ as a curve in $F_2$. Consider the linear system $V' = |3e + 6f|$ on $F_2$. Let $U' \subset V'$ be the open locus of smooth curves. Let $NI'_1\subset V'$ be the codimension one locus of irreducible nodal curves. Let $NI'_2\subset V'$ be the codimension one locus of irreducible at-worst-nodal curves $B$ of class $2e+ 6f$ union $E$, where $B$ and $E$ intersect transversally at two distinct points. Let $NI'_3\subset V'$ be the codimension two locus of irreducible at-worst-nodal curves $B$ of class $2e+ 6f$ union $E$, where $B$ and $E$ are tangent at one point. Let 
$CP'\subset V'$ be the codimension two locus of irreducible curves with a cusp. Note that $V'\backslash (U'\cup NI'_1\cup NI'_2\cup NI'_3\cup CP')$ has codimension three in $V'$. Curves in $U'\cup NI'_1$ are stable and curves in $CP'$ are pseudostable. For a curve in $NI'_2$, its stablization as a curve in $Q$ has a node at $v$. For a curve in $NI'_3$, its pseudostablization as a curve in $Q$ has a cusp at $v$. In other words, blowing down $E$, curves in $NI'_2$ and in $NI'_3$ become stable and pseudostable, respectively. 

Let $h$ be the class of $H$. 
Suppose that 
\begin{eqnarray}
\label{eq:hyp-4}
h = \sum_i a_i [Z_i] \in N^2(\BM_4) 
\end{eqnarray}
for $a_i > 0$ and $Z_i\subset \BM_4$ irreducible subvariety of codimension two, not equal to $H$. 
Recall the first divisorial contraction $\ps: \BM_4\to \BM_{4}^{\ps}$ induced by replacing an elliptic tail by a cusp for curves in $\Delta_1$. Applying $\ps_{*}$, we obtain that 
\begin{eqnarray}
\label{eq:hyp-ps} 
\ps_{*}h = \sum_i a_i (\ps_{*}[Z_i]) \in N^2(\BM_4^{\ps}). 
\end{eqnarray}

We first show that $Z_i$ is contained in $GP\cup \Delta$ for all $i$. If $\ps_{*}[Z_i] = 0$, then $Z_i \subset \Delta_1$. Hence we may assume that $\ps_{*}[Z_i] \neq 0$ in \eqref{eq:hyp-ps}. Suppose that a general curve $[C]\in Z_1$ is not contained in $GP\cup \Delta$. Then $C$ is smooth and its canonical embedding is contained in a smooth quadric surface. 
Take a net $S$ in $V$ such that $S$ is spanned by $C$ and two other general $(3,3)$-curves. Then $S$ is disjoint from $V\backslash (U\cup NI \cup CP)$. Moreover, any curve parameterized in $S$ is pseudostable and not contained in $\ps(GP)$. For a curve $[D]\in \ps(Z_i)$, if $D$ fails to yield a canonical embedding lying on a smooth quadric, then $[D] \not \in S$. The locus of $\ps(Z_i)$ corresponds to a locus of codimension $\geq 2$ in $V$. 
Therefore, $[S]\cdot (\ps_{*}[H]) = 0$, $[S]\cdot (\ps_{*}[Z_1]) > 0$ and $[S]\cdot (\ps_{*}[Z_i]) \geq 0$ for $i \neq 1$. Plugging in \eqref{eq:hyp-ps} leads to a contradiction. We thus conclude that $Z_i \subset GP\cup \Delta$ for all $i$.  

Next, we show that if $Z_i \not= H$, then it is contained in $\Delta$. 
Suppose that a general curve $[C]\in Z_1\neq H$ is contained in $GP$ but not in $\Delta$. In particular, $C$ is smooth and non-hyperelliptic. 
Take a net $S$ of quadric surfaces in $\bbP^3$ such that it possesses a unique quadric cone $Q$ containing the canonical embedding of $C$ and is general other than that. Let $T$ be a cubic surface such that $C = T\cap Q$. The intersection of $T$ with surfaces parameterized in $S$ gives a two-dimensional family of curves in $\BM_4$, still denoted by $S$. There is a one-dimensional family of quadric cones in $S$, hence $S$ intersects $GP$ along a one-dimensional family $B$. Curves in $B$ become most degenerate when the vertex of the cone meets $T$. In this case the curve is a nodal curve contained in $GP$ such that the node coincides with the vertex, i.e. it belongs to the locus $NI'_2$ and is stable. We see that curves in $S$ are stable and non-hyperelliptic. Moreover, $B$ intersects  
$Z_1$ at finitely many points including $[C]$ and can be made general by the choice of $S$ to avoid any other locus of codimension two or higher in $V'$. 
It follows that $[S]\cdot h = 0$, $[S]\cdot [Z_1] > 0$ and $[S]\cdot [Z_i] \geq 0$ for $i \neq 1$. Plugging in \eqref{eq:hyp-4} leads to a contradiction. We thus conclude that if $Z_i \not= H$, then $Z_i$ is contained in $\Delta$. 

Now it follows from \eqref{eq:hyp-4} that $h = 0 \in N^2(\MM_4)$, which contradicts the fact that $h$ is a nonzero multiple of $\lambda^2$ in $\MM_4$ (\cite{FaberChow4}). 
\end{proof}



\section{The effective cones of $\BM_{0,n}$}
\label{sec:eff0n}

In this section, we study effective cycles on $\BM_{0,n}$. Since $\dim \BM_{0,n} = n-3$, there are no interesting higher codimension cycles for $n \leq 6$. Hence, from now on we assume that $n\geq 7$. 

Let $S$ be a subset of $\{1, \ldots, n\}$ such that $|S|, |S^c| \geq 2$. To make the notation symmetric, denote by $\Delta_{S, S^c}$ 
the boundary divisor of $\BM_{0,n}$ 
parameterizing genus zero curves that have a node that separates the curve into two curves,   
 one marked by $S$ and the other by $S^c$. By definition, $\Delta_{S, S^c} = \Delta_{S^c, S}$. The codimension one boundary strata of $\BM_{0,n}$ consist of the divisors $\Delta_{S, S^c}$, where the pair $\{S, S^c\}$ varies over subsets of $\{1,\ldots,n\}$ with $|S|, |S^{c}| \geq 2$. 

Consider the codimension two boundary strata of $\BM_{0,n}$. Let $S_1, S_2, S_3$ be an \emph{ordered} decomposition of $\{1,\ldots, n  \}$ with $s_i = |S_i|$ such that $s_1 + s_2 + s_3 = n$, 
$s_1, s_3 \geq 2$ and $s_2 \geq 1$. Let $D_{S_1, S_2, S_3}$ be the codimension two boundary stratum of $\BM_{0,n}$
whose general point parameterizes a chain of three smooth rational curves $C_1, C_2, C_3$ such that $C_i$ is marked by $S_i$. By definition, $D_{S_1, S_2, S_3} = D_{S_3, S_2, S_1}$. 

\begin{theorem}
\label{thm:eff0n}
The cycle class of $D_{S_1, S_2, S_3}$ is extremal in $\Eff^2(\BM_{0,n})$. 
\end{theorem}

\begin{proof}
By Proposition~\ref{prop:main-tool}, it suffices to exhibit morphisms $f$ such that $e_{f}(D_{S_1, S_2, S_3})>e_f(Z)$ for any subvariety $Z$ of codimension two that is not $D_{S_1, S_2, S_3}$. We will use the morphisms 
$f_{\calA}: \BM_{0,n} \to \BM_{0, \calA}$ introduced in Section~\ref{sec:prelim-moduli}. 
 
First, suppose $s_1, s_3 > 2$. 
Let $\calA$ be the weight parameter assigning $\frac{1}{s_1}$ to the marked points in $S_1$ and $1$ to the other marked points. Then $e_{f_{\calA}}(D_{S_1, S_2, S_3}) = s_1 - 2>0$. By Proposition~\ref{prop:main-tool}, 
if $e_{f_{\calA}}(Z)\geq s_1 - 2$, then $Z$ has to be contained in $\Delta_{S_1, S_1^{c}}$. 
Similarly, let $\calB$ assign $\frac{1}{s_3}$ to the marked points in $S_3$ and $1$ to the other marked points. 
Then $e_{f_{\calB}}(D_{S_1, S_2, S_3}) = s_3 - 2>0$. 
If $e_{f_{\calB}}(Z)\geq s_3 - 2$, then $Z$ is contained in $\Delta_{S_3, S_3^c}$. Since the intersection of $\Delta_{S_1, S_1^{c}}$ and $\Delta_{S_3, S_3^c}$ 
is exactly $D_{S_1, S_2, S_3}$, we conclude that  $D_{S_1, S_2, S_3}$ is extremal when $s_1, s_3 > 2$. 

Next, suppose that $s_1 = s_3 = 2$.  Since $n\geq 7$, we have $s_2 = n - 4 \geq 3$. Without loss of generality, 
let $S_1 = \{1,2\}$, $S_3 = \{3,4 \}$ and $S_2 = \{ 5, \ldots, n \}$. Let $\calC$  assign 
$\frac{1}{n-2}$ to the marked points in $S_1\cup S_2$ and $1$ to the marked points in $S_3$. 
We have $e_{f_{\calC}}(D_{S_1, S_2, S_3}) = s_1 + s_2 - 3 = n-5$. If $e_{f_{\calC}}(Z)\geq n-5$, 
then $Z$ is contained in $\Delta_{S, S^c}$, where $S\subset \{1,2, 5, \ldots, n\}$ and $|S| = n-2$ or $n-3$. 
Similarly let $\calD$ assign $\frac{1}{n-2}$ to the marked points in $S_3\cup S_2$ and $1$ to the marked points in $S_1$. 
We have $e_{f_{\calD}}(D_{S_1, S_2, S_3}) = s_2 + s_3 - 3 = n-5$. 
If $e_{f_{\calD}}(Z)\geq n-5$, then $Z$ is contained in $\Delta_{T, T^c}$ where $T\subset \{3,4, 5, \ldots, n\}$ and $|T| = n-2$ or $n-3$. 
When $|S| = |T| = n-2$, the intersection of $\Delta_{S, S^c}$ and $\Delta_{T, T^c}$ is exactly $D_{S_1, S_2, S_3}$. If 
$|S| = n-2$ and $|T| = n-3$, then the intersection of $\Delta_{S, S^c}$ and $\Delta_{T, T^c}$ is of the type 
$D_{S'_1, S'_2, S'_3}$ where $S'_1 = \{3,4\}$ and $S'_3 = \{1, 2, i\}$ for $i\in S_2$. We have 
$e_{f_{\calD}}(D_{S'_1, S'_2, S'_3}) = n-6 <  e_{f_{\calD}}(D_{S_1, S_2, S_3}) $. Finally, if 
$|S| = |T| = n-3$, the intersection of $\Delta_{S, S^c}$ and $\Delta_{T, T^c}$ is of the type 
$D_{S'_1, S'_2, S'_3}$ where $S'_1 = \{3,4, i\}$ and $S'_3 = \{1, 2, j\}$ for $i\neq j \in S_2$. 
We have $e_{f_{\calD}}(D_{S'_1, S'_2, S'_3}) = n-6 < e_{f_{\calD}}(D_{S_1, S_2, S_3})$.
This proves that  $D_{S_1, S_2, S_3}$ is extremal when $s_1= s_3 = 2$. 

The case $s_1 = 2$ and $s_3 > 2$ can be verified by a similar (and simpler) argument as in the above paragraph, so we omit the details. 
\end{proof}

 Next, we give explicit examples of extremal higher codimension cycles in $\BM_{0,n}$. Let 
$S_1, \ldots, S_k$ be an \emph{unordered} decomposition of $\{1, \ldots, n\}$ such that 
$k\geq 3$ and $|S_i|\geq 2$ for each $i$. Denote by $B_{S_1, \ldots, S_k}$ the subvariety of $\BM_{0,n}$ whose general point parameterizes a rational curve $R$ attached to $k$ rational tails $C_1,\ldots, C_k$ such that $C_i$ is marked by $S_i$. The codimension of $B_{S_1, \ldots, S_k}$ in $\BM_{0,n}$ equals $k$. 

\begin{theorem}
\label{thm:unmarked}
The cycle class of $B_{S_1, \ldots, S_k}$ is extremal in $\Eff^k(\BM_{0,n})$. 
\end{theorem}

\begin{proof}
Suppose that $[B_{S_1, \ldots, S_k}] = \sum_{j=1}^{r} a_j[Z_j] \in N^k(\BM_{0,n})$, 
where $a_j > 0$ and $Z_j\subset \BM_{0,n}$ is an irreducible subvariety of codimension $k$. First, consider the case $|S_i| = s_i > 2$ for all $i$. 
Let $\calA_i$ assign $\frac{1}{s_i}$ to the marked points in $S_i$ and $1$ to the other marked points. Then $e_{f_{\calA_i}}(B_{S_1, \ldots, S_k}) = s_i -2 > 0$. 
If $e_{f_{\calA_i}}(Z) \geq s_i -2$, then $Z$ has to be contained in $\Delta_{S_i, S_i^c}$. Hence by Proposition~\ref{prop:main-tool}, $Z_j$ is contained in $\Delta_{S_1,S_1^c}\cap \cdots \cap \Delta_{S_k,S_k^c}$ for all $j$. The intersection locus is exactly $B_{S_1, \ldots, S_k}$, which is irreducible. We thus conclude that $Z_j = B_{S_1, \ldots, S_k}$. 

Next, consider the case $k\geq 4$ and some $s_i = 2$. Without loss of generality, assume that $S_{i} = \{2i-1, 2i \}$ for $1\leq i \leq l$ and $s_{l+1},\ldots, s_k > 2$. Let $\overline{\calB}_{0,n}$ be the moduli space of $n$-pointed genus zero curves with rational $k$-fold singularities (without unmarked components). A rational $k$-fold singularity is locally isomorphic to the intersection of the $k$ concurrent coordinate axes in 
$\mathbb A^k$. There is a birational morphism $f: \BM_{0,n}\to \overline{\calB}_{0,n}$ contracting unmarked components $X$ to a rational $k$-fold singularity, where $k = \# (X\cap 
\overline{C\backslash X})\geq 3$, see e.g. \cite[3.7--3.11]{ChenCoskunKont}. Note that 
$e_f(B_{S_1, \ldots, S_k}) = k - 3 > 0$. By Proposition~\ref{prop:main-tool}, we have $e_f(Z_j) \geq k -3$. Since 
an unmarked component with $m$ nodes loses $(m-3)$-dimensional moduli under $f$, we conclude that 
$Z_j$ has to be one of the $B_{T_1, \ldots, T_k}$. Using the morphism $f_{\calA_i}$ in the previous paragraph  
for $i > l$, we further conclude that $Z_j$ is of the type  
$B_{T_1, \ldots, T_l, S_{l+1}, \ldots, S_k}$, where $T_1, \ldots, T_l$ yield a decomposition of $\{ 1, \ldots, 2l\}$. Since 
$|T_i| \geq 2$, it implies that $|T_i| = 2$ for each $i$. Therefore, $Z_j$ is the image of $B_{S_1, \ldots, S_k}$ 
under the automorphism of $\BM_{0,n}$ induced by relabeling $\{2i-1, 2i \}$ as $T_i$ for each $i$. By symmetry,  
any one of them cannot be a nonnegative linear combination of the others. 
 
The remaining case is that $k=3$ and some $s_i=2$. Since $n\geq 7$, without loss of generality, assume that 
$S_1 = \{ 1,2\}$, $S_2 = \{3, \ldots, m \}$ and $S_3 = \{m+1, \ldots, n \}$ with $4\leq m \leq n-3$ so that $s_3 = n-m > 2$. 
As above, we know that 
$Z_j$ is contained in $\Delta_{S_3, S_3^c}$. By \cite[p. 549]{Keel}, we know that 
 $N^2(\Delta_{S_3, S_3^c})\to N^3(\BM_{0,n})$ is injective. Moreover, $B_{S_1, S_2, S_3}$ can be identified with 
 $D_{S_1, \{p\}, S_{2}} \times \BM_{0,n-m+1}\subset \Delta_{S_3, S_3^c}\cong \BM_{0,m+1}\times \BM_{0,n-m+1}$, 
 where $p$ denotes the node in the rational tail marked by $S_3$. By Theorem~\ref{thm:eff0n}, $D_{S_1, \{p\}, S_{2}}$ 
  is an extremal codimension two cycle in $\BM_{0, m+1}$, hence $B_{S_1, S_2, S_3}$ is an extremal codimension two cycle in $\Delta_{S_3, S_3^c}$ by Corollary~\ref{cor:prod}. It follows that $B_{S_1, S_2, S_3}$ is an extremal codimension three cycle in $\BM_{0,n}$ by Proposition~\ref{prop:subvariety}. 
\end{proof}

Finally, we point out that not all boundary strata are extremal cycles on moduli spaces of pointed stable curves. 

\begin{remark}
Recall that $\widetilde{\MM}_{0,n} = \BM_{0,n}/\mathfrak{S}_n$ is the moduli space of stable genus zero curves with $n$ unordered marked points. The one-dimensional
strata of $\widetilde{\MM}_{0,n}$ ($F$-curves) correspond to partitions of $n$ into four parts. For $n= 7$, there are three $F$-curves on 
$\widetilde{\MM}_{0,7}$: $F_{1,1,1,4}$, $F_{1,1,2,3}$ and $F_{1,2,2,2}$. However, their numerical classes satisfy that 
$2 [F_{1,1,2,3}] =  [F_{1,1,1,4}] + [F_{1,2,2,2}]$ (see e.g. \cite[Table 1]{Moon}). In particular, $F_{1,1,2,3}$ is not extremal in the 
Mori cone of curves on $\widetilde{\MM}_{0,7}$. 
\end{remark}

\section{The effective cones of $\BM_{1,n}$}
\label{sec:eff1n}

In this section, using the infinitely many extremal effective divisors in $\BM_{1,n-2}$ constructed in \cite{ChenCoskun}, we show that $\Eff^2(\BM_{1,n})$ is not finite polyhedral for $n\geq 5$. 


Let  $p_{1}, \ldots, p_n$ be the $n$ marked points and let $T = \{p_{n-2}, p_{n-1}, p_n\}$. 
The gluing morphism 
$$\HD_{0; T}=\BM_{1,n-2} \times \BM_{0,4}\to \Delta_{0; T}\subset \BM_{1,n}$$ 
 is induced by gluing a pointed genus one curve $(E, p_1, \ldots, p_{n-3}, q)$ to a pointed rational curve $(C, p_{n-2}, p_{n-1}, p_n, q)$ by identifying $q$ in  $E$ and $C$ to form a node. Denote by $\Gamma_{S}$ (resp. $\Gamma_{0})$ the image of $\Delta_{0;S}\times \BM_{0,4}$ (resp. $\Delta_{0} \times \BM_{0,4}$) in $\BM_{1,n}$ for $S\subset \{ p_1, \ldots, p_{n-3}, q \}$ and $|S| \geq 2$. Let $\Gamma$ be the image of $ \BM_{1,n-2}\times \Delta_{0; \{p_{n-1}, p_n\}}$. Note that $\Gamma_{0}$, $\Gamma_{S}$ and $\Gamma$ are the codimension two boundary strata of $\BM_{1,n}$ whose general point parameterizes a two-nodal curve that has a rational tail marked by $T$. The cases $q\in S$ and $q\not \in S$ 
correspond to the middle component having genus zero and one, respectively. 

Since $\BM_{0,4}\cong \bbP^1$, it follows that $A^1(\HD_{0; T}) \cong N^1(\HD_{0; T})$, generated by 
the classes of $\Gamma_S$, $\Gamma_{0}$ and $\Gamma$. 

\begin{lemma}
\label{lem:embed-1}
The cycle classes of $\Gamma_S$, $\Gamma_{0}$ and $\Gamma$ are independent in $N^2(\BM_{1,n})$. 
\end{lemma}

\begin{proof}
Denote by $\gamma_{S}$, $\gamma_{0}$ and $\gamma$ the classes of $\Gamma_S$, $\Gamma_{0}$ and $\Gamma$, respectively. 
Suppose that they satisfy a relation
\begin{equation}
\label{eq:relation-1-1}
 a \cdot \gamma_{0} + \sum_S b_S \cdot \gamma_S + c \cdot \gamma = 0 
 \end{equation}
 in $N^2(\BM_{1,n})$. Below we will show that all the coefficients in \eqref{eq:relation-1-1} are zero. 

Let $\calA$ be the weight parameter assigning $\frac{1}{3}$ to $p_{n-2}, p_{n-1}, p_n$ and $1$ to $p_k$ 
for $k \leq n-3$. Then $e_{f_{\calA}}(\Gamma) = 0 $ and $e_{f_{\calA}}(\Gamma_0) =  e_{f_{\calA}}(\Gamma_S) = 1$. Applying $f_{\calA_{*}}$ to \eqref{eq:relation-1-1}, we conclude that $c=0$. Relation \eqref{eq:relation-1-1} reduces to 
\begin{equation}
\label{eq:relation-1-2}
 a\cdot \gamma_{0} + \sum_S b_S\cdot \gamma_S = 0. 
 \end{equation}
 
Let $\psi: \BM_{1,n}\to \BM_{1,n-2}$ be the morphism forgetting $p_1$ and $p_2$. We have 
$e_{\psi}(\Gamma_{\{p_1, p_2\}}) = 0 $ and $e_{\psi}(\Gamma_0), e_{\psi}({\Gamma_S}) > 0$ for $S\neq \{p_1, p_2\}$. 
Applying $\psi_{*}$ to \eqref{eq:relation-1-2}, we conclude that $b_{\{p_1, p_2\}} = 0$, hence by symmetry 
$b_S = 0$ for $S = \{p_i, p_j \}$ for $i, j \leq n-3$. Next, let $\phi: \BM_{1,n} \to \BM_{1,n-1}$ be the morphism forgetting $p_1$. Among the remaining cycles, $e_{\phi}(\Gamma_{\{p_1, q\}}) = 0$ and $e_{\phi}(\Gamma_0) = e_{\phi}(\Gamma_{S}) > 0$ for $S\neq \{p_1, q \}$. Applying $\phi_{*}$, we conclude that $b_{\{p_1, q\}} = 0$, hence 
by symmetry $b_{\{p_i, q\}} = 0$ for all $i\leq n-3$. In sum, we obtain that $b_S = 0$ for all $|S| = 2$. 

Apply induction on $|S|$. Suppose that $b_S = 0$ for all $|S| \leq k-1$. Relation~\eqref{eq:relation-1-2} reduces to 
\begin{equation}
\label{eq:relation-1-k}
 a \cdot \gamma_{0} + \sum_{|S| \geq k} b_S \cdot \gamma_S = 0. 
 \end{equation}
Let $\varphi: \BM_{1,n}\to \BM_{1, n-k}$ be the morphism forgetting $p_1, \ldots, p_k$. Then we have 
$e_{\varphi}(\Gamma_{\{p_1,\ldots, p_k\}})  = k-2$ and $e_{\varphi}(\Gamma_0), e_{\varphi}(\Gamma_S) > k-2$ for 
$|S| \geq k$ and $S\neq \{p_1, \ldots, p_k\}$. Take an ample divisor class $A$ in $\BM_{1,n}$, 
intersect \eqref{eq:relation-1-k} with $A^{k-2}$ and apply $\varphi_{*}$. We obtain that $b_{\{p_1, \ldots, p_k\}} = 0$, hence by symmetry $b_S = 0$ for $S\subset \{p_1, \ldots, p_{n-3}\}$ and $|S| = k$. Next, let $\eta: \BM_{1,n} \to \BM_{1,n-k+1}$ be the morphism forgetting $p_1, \ldots, p_{k-1}$. Among the remaining cycles, 
$e_{\eta}(\Gamma_{\{p_1, \ldots, p_{k-1}, q\}}) = k-2$ and $e_{\eta}(\Gamma_0), e_{\eta}(\Gamma_S) \geq k-1$
for $S\neq \{p_1, \ldots, p_{k-1}, q\}$. Intersecting with $A^{k-2}$ and applying $\eta_{*}$, we obtain that 
$b_{\{p_1, \ldots, p_{k-1},q\}} = 0$, hence $b_S = 0$ for all $|S| = k$. By induction we thus conclude that 
$b_S = 0$ for all $S$. 

Finally, Relation~\eqref{eq:relation-1-k} reduces to $a \cdot \gamma_ 0 = 0$, hence $a = 0$. 

\end{proof}

\begin{theorem}
\label{thm:eff1n}
For $n \geq 5$, $\Eff^2(\BM_{1,n})$ is not finite polyhedral. 
\end{theorem}

\begin{proof} 
Let $\calA$ be the weight parameter assigning $\frac{1}{3}$ to marked points in $T = \{ p_{n-2}, p_{n-1}, p_n\}$ and $1$ to the other marked points. The exceptional locus of  $f_{\calA}: \BM_{1,n} \to \BM_{1,\calA}$ is $\Delta_{0; T}$. By Corollary~\ref{cor:prod}, any extremal effective divisor $D'$ on $\BM_{1,n-2}$ pulls back to an extremal divisor $D$ on $\HD_{0;T}=\BM_{1,n-2}\times \BM_{0,4}$. Moreover, $e_{f_{\calA}}(D) > 0$ because the moduli of the rational tail marked by $T$ is forgotten under $f_{\calA}$. By Lemma~\ref{lem:embed-1}, we can apply Proposition~\ref{prop:subvariety} to conclude that the class of $D$ is extremal in 
$\Eff^2(\BM_{1,n})$. Now the claim follows from the fact that there are infinitely many extremal effective divisors on $\BM_{1,n-2}$ for every $n\geq 5$ (\cite{ChenCoskun}). 
\end{proof}

\section{The effective cones of $\BM_{2,n}$}
\label{sec:eff2n}

Applying the same idea as in Section~\ref{sec:eff1n}, in this section we show that $\Eff^2(\BM_{2,n})$ is not finite polyhedral for $n\geq 2$. 


Let $p_1, \ldots, p_n$ be the $n$ marked points. The gluing morphism 
$$\HD_{1; \emptyset}=\BM_{1,n+1} \times 
\BM_{1,1} \to \Delta_{1; \emptyset}\subset \BM_{2,n}$$
is induced by gluing two pointed genus one curves $(E_1, p_1, \ldots, p_{n}, q)$ and $(E_2, q)$ by identifying $q$ in both curves to form a node. Denote by $\Gamma_{S}$ (resp. $\Gamma_{0})$ the image of $\Delta_{0;S}\times \BM_{1,1}$ (resp. $\Delta_{0} \times \BM_{1,1}$) in $\BM_{2,n}$ for $S\subset \{ 1, \ldots, n+1 \}$ and $|S| \geq 2$. Let $\Gamma$ be the image of $\BM_{1,n+1}\times \Delta_0$. Note that $\Gamma_{0}$, $\Gamma_{S}$ and $\Gamma$ are the codimension two boundary strata of $\BM_{2,n}$ whose general point parameterizes a curve with two nodes and an unmarked tail of arithmetic genus one. 
The cases $q\in S$ and $q\not \in S$ 
correspond to the middle component having genus zero and one, respectively. 

Since $\BM_{1,1}\cong \bbP^1$, it follows that $A^1(\HD_{1; \emptyset}) \cong N^1(\HD_{1; \emptyset})$, generated by 
the classes of $\Gamma_S$, $\Gamma_{0}$ and $\Gamma$. 

\begin{lemma}
\label{lem:embed}
The cycle classes of $\Gamma_S$, $\Gamma_{0}$ and $\Gamma$ are independent in $N^2(\BM_{2,n})$. 
\end{lemma}

\begin{proof}
Denote by $\gamma_{S}$, $\gamma_{0}$ and $\gamma$ the classes of $\Gamma_S$, $\Gamma_{0}$ and $\Gamma$ in $\BM_{2,n}$, respectively. 
Suppose that they satisfy 
\begin{equation}
\label{eq:relation-2}
 a \cdot \gamma_{0} + \sum_S b_S\cdot  \gamma_S + c \cdot \gamma = 0. 
 \end{equation}
Recall that $\ps: \BM_{2,n} \to \BM_{2,n}^{\ps}$ contracts $\Delta_{1; \emptyset}$, replacing an elliptic tail by a cusp. Hence we have $e_{\ps}(\Gamma) = 0$ and 
$e_{\ps}(\Gamma_0) = e_{\ps}(\Gamma_S) = 1$ for all $S$. Applying $\ps_{*}$ to \eqref{eq:relation-2}, we conclude that 
$c = 0$. Now the same induction procedure in the proof of Lemma~\ref{lem:embed-1} implies that $a = b_S  = 0$ for all $S$.
\end{proof}

\begin{theorem}
\label{thm:eff2n}
For $n \geq 2$, $\Eff^2(\BM_{2,n})$ is not finite polyhedral. 
\end{theorem}

\begin{proof}
The exceptional locus of $\ps: \BM_{2,n} \to \BM_{2,n}^{\ps}$ is 
$\Delta_{1; \emptyset}$. By Corollary~\ref{cor:prod}, any extremal effective divisor $D'$ on $\BM_{1,n+1}$ pulls back to an extremal divisor $D$ on $\HD_{1;\emptyset}=\BM_{1,n+1}\times \BM_{1,1}$. Moreover, $e_{\ps}(D) > 0$ because the unmarked genus one tail is forgotten under $\ps$. By Lemma~\ref{lem:embed}, we can apply Proposition~\ref{prop:subvariety} to conclude that the class of $D$ is extremal in 
$\Eff^2(\BM_{2,n})$. Now the claim follows from the fact that there are infinitely many extremal effective divisors on $\BM_{1,n+1}$ for every $n\geq 2$ (\cite{ChenCoskun}). 
\end{proof}

\begin{remark}
Since $N^1(\HD_1)\cong A^1(\HD_1)  \to N^2(\BM_g)$ is injective and $\HD_1 \cong \BM_{g-1,1}\times \BM_{1,1}$, the same proof as that of Theorem~\ref{thm:eff2n} implies that the pullback of any extremal effective divisor from $\BM_{g-1,1}$ to $\HD_1$ gives rise to an extremal codimension two cycle in $\BM_g$. 
For instance, the divisor $W$ of Weierstrass points is known to be extremal in $\BM_{g-1,1}$ for $3\leq g\leq 6$ 
(see e.g. \cite[Theorem 4.3 and Remark 4.4]{ChenCycle}). Hence $W$ yields an extremal codimension two cycle in $\Eff^2(\BM_g)$ for $3\leq g\leq 6$.  
\end{remark}

\end{document}